\documentclass[11pt]{amsart}
\usepackage{amsmath,amssymb,amsthm}
\usepackage{graphicx}
\usepackage{color}
\usepackage{hyperref}
\usepackage{float}

\newtheorem{lem}{Lemma}[section]
\newtheorem{prop}{Proposition}[section]

\newtheorem{thm}{Theorem}[section]

\theoremstyle{definition}

\theoremstyle{remark}

\theoremstyle{remark}
\newtheorem{remark}{Remark}[section]
\numberwithin{equation}{section}

\newcommand{\N}{{\mathbb N}}

\newcommand{\R}{{\mathbb R}}

\definecolor{blu}{rgb}{0,0,1}

\title[Multiple normalized solutions for quasi-linear Schr\"odinger equations]{Multiple normalized solutions for quasi-linear Schr\"odinger equations}

\author[Louis Jeanjean]{Louis Jeanjean}
\address{Louis Jeanjean
\newline\indent
Laboratoire de Math\'ematiques (UMR 6623)
\newline\indent
Universit\'{e} de Franche-Comt\'{e}
\newline\indent
16 Route de Gray, 25030  Besan\c{c}on Cedex, France}
\email{louis.jeanjean@univ-fcomte.fr}

\author[Tingjian Luo]{Tingjian Luo}
\address{Tingjian Luo
\newline\indent
School of Mathematics and Information Science
\newline\indent
Guangzhou University
\newline\indent
Guangzhou 510006, P.R. China and
\newline\indent
Laboratoire de Math\'ematiques (UMR 6623)
\newline\indent
Universit\'{e} de Franche-Comt\'{e}
\newline\indent
16 Route de Gray, 25030  Besan\c{c}on Cedex, France}
\email{tingjianluo@gmail.com}

\author[Zhi-Qiang Wang]{Zhi-Qiang Wang}
\address{Zhi-Qiang Wang
\newline\indent
Center for Applied Mathematics
\newline\indent
Tianjin University
\newline\indent
Tianjin 300072, P.R. China and
\newline\indent
Department of Mathematics and Statistic
\newline\indent
Utah State University
\newline\indent
Logan, UT84322, USA}
\email{zhi-qiang.wang@usu.edu}

\begin{document}


\begin{abstract} In this paper we prove the existence of two solutions having a prescribed $L^2$-norm for a quasi-linear Schr\"odinger equation. One of these solutions is a mountain pass solution relative to a constraint and the other one a minimum either local or global. To overcome the lack of differentiability of the associated functional, we rely on a perturbation method developed in \cite{LLW}.
\end{abstract}
\maketitle

\section{Introduction}
In this paper, we are concerned with quasi-linear Schr\"odinger equations of the form
\begin{equation}\label{eq1.1}
\left\{
\begin{array}{l}
i \partial_t \varphi+\Delta \varphi +\varphi \Delta (|\varphi|^2)+|\varphi|^{p-1}\varphi=0,\ \mbox{ in } \R^+\times \R^N,\\
\varphi(0,x)=\varphi_0(x),\ \mbox{ in } \R^N,
\end{array}
\right.
\end{equation}
where $p\in (1, \frac{3N+2}{N-2})$ if $N\geq 3$ and $p\in (1, \infty)$ if $N=1,2$, $i$ denotes the imaginary unit and the unknown $\varphi: \R^+\times \R^N \rightarrow \mathbb{C}$ is a complex valued function. Such types of equations appear in various physical fields, for instance in dissipative quantum mechanics, in plasma physics and in fluid mechanics. We refer the readers to \cite{CJS,PSW} and their
 references for more information on the related physical backgrounds.
\medskip

From the physical as well as the mathematical point of view, a central issue is the existence and dynamics of standing waves of \eqref{eq1.1}. By standing waves, we mean solutions of the form $\varphi(t,x)= e^{-i \lambda t}u(x)$, where $\lambda\in \R$ is a parameter. Observe that $e^{-i \lambda t}u(x)$ solves \eqref{eq1.1} if and only if $u(x)$ satisfies the following stationary equation
\begin{equation}\label{eq1.2}
- \Delta u - u\Delta (|u|^2)-\lambda u - |u|^{p-1}u=0,\ \mbox{ in }  \R^N.
\tag{$P_{\lambda}$}
\end{equation}

In \eqref{eq1.2}, when $\lambda\in \R$ appears as a fixed parameter, the existence and multiplicity of solutions of \eqref{eq1.2} have been intensively studied during the last decade. See \cite{AZ, CLW, CJ, CJS, FS, LW, LW1, LWW1, LWW2, LLW0, LLW, LLW1, PSW, RS} and their references therein. We also refer to \cite{AT, AZ, GS, AS} for the uniqueness of ground states of \eqref{eq1.2}. By a ground state we mean
a  solution of \eqref{eq1.2} which minimize among all  nontrivial solutions the associated energy functional
\begin{equation*}
I_{\lambda}(u):=\frac{1}{2}\int_{\R^N}|\nabla u|^2dx - \frac{\lambda}{2}\int_{\R^N}|u|^2dx   +\int_{\R^N}|u|^2|\nabla u|^2dx-\frac{1}{p+1}\int_{\R^N}|u|^{p+1}dx,
\end{equation*}
defined on the natural space
$$\mathcal{X}:= \Big\{u\in W^{1,2}(\R^N):\ \int_{\R^N}|u|^2|\nabla u|^2 dx<\infty \Big\}.$$
It is easy to check that $u$ is a weak solution  of \eqref{eq1.2} if and only if
$$
I_{\lambda}'(u) \phi := \lim_{t \to 0^+}\frac{I_{\lambda}(u+t\phi) - I_{\lambda}(u)}{t}=0,
$$
for every direction $\phi \in C_0^{\infty}(\R^N,\R)$. We also recall, see \cite[Remark 1.7]{CJS}  and \cite{LWW1} for example, that when $N\geq 3$ the value $\frac{3N+2}{N-2}$ corresponds to a critical exponent.\medskip

Compared to semi-linear equations where the term $u\Delta (|u|^2)$ is not present, the search of solutions of \eqref{eq1.2} presents a major difficulty. The functional associated with the quasi-linear term
$$V(u):= \int_{\R^N}|u|^2|\nabla u|^2dx,$$
is non differentiable in the space $\mathcal{X}$ when $N \geq 2$.
To overcome this difficulty, various arguments have been developed. First in \cite{LW, PSW}, solutions of \eqref{eq1.2} are
obtained by minimizing the functional $I_{\lambda}$ on the set
$$\Big\{u\in \mathcal{X}:\ \int_{\R^N}|u|^{p+1}dx = 1 \Big\}.$$
In the proofs of \cite{LW, PSW} the non-differentiability of $I_{\lambda}$ essentially does not come into play. Alternatively in \cite{CJ, LWW1}, by a change of unknown, the quasi-linear problem \eqref{eq1.2} is transformed into a semi-linear problem. For that semi-linear problem standard variational methods can be applied to yield a solution. Also in \cite{LWW2} the authors have developed an approach which works for more general quasi-linear equations and which one reduces the search of solutions of \eqref{eq1.2} to the problem of showing that
$I_{\lambda}$ has a global minimizer on a Nehari manifold. Since these pioneering works, there has been a large literature on the study of the equation \eqref{eq1.2} where are addressed questions of multiplicity, concentration type as well as critical exponent issues. In particular, starting from \cite{LLW} in a series of papers \cite{LW1, LLW0, LLW, LLW1} Liu, Liu and Wang have developed a perturbation method for studying existence and multiplicity of solutions for a general class of quasi-linear elliptic equations including the above discussed model equations.\medskip

In the present paper, motivated by the fact that physicists are often interested in ``normalized solutions" we look for solutions to \eqref{eq1.2} having a prescribed $L^2$-norm. In this aim, for given $c>0$, one can look to minimizers of
\begin{equation}\label{mini1}
m(c):= \inf_{u\in S(c)}J(u),
\end{equation}
where
\begin{equation}\label{set}
S(c):= \Big\{u\in \mathcal{X}:\  \int_{\R^N}|u|^2dx=c \Big\}.
\end{equation}
Here the functional $J: S(c)\rightarrow \R$, is defined as
\begin{eqnarray}\label{func1}
J(u):=\frac{1}{2}\int_{\R^N}|\nabla u|^2dx+\int_{\R^N}|u|^2|\nabla u|^2dx-\frac{1}{p+1}\int_{\R^N}|u|^{p+1}dx.
\end{eqnarray}
It is proved in \cite[Theorem 4.6]{CJS}) that each minimizer $u\in S(c)$ of \eqref{mini1}, corresponds a Lagrange multiplier $\lambda<0$ such that $(u, \lambda)$ solves weakly \eqref{eq1.2}.
\medskip

We collect below the known results concerning the function $c \to m(c)$ and the minimizers of $m(c)$.
\begin{lem}\label{lm0} (\cite[Theorems 1.9, 1.12]{CJS}, \cite[Theorems 1.4, 1.5]{JL2})
Assume that $p\in (1, \frac{3N+2}{N-2})$ if $N\geq 3$ and $p\in (1, \infty)$ if $N=1,2$. Then
\smallskip

\begin{itemize}
  \item [(1)]
    \begin{itemize}
      \item [i)] For all $c>0$, $m(c)\in (-\infty, 0]$ as $p\in (1, 3+\frac{4}{N})$;
      \item [ii)] For all $c>0$, $m(c)=-\infty$ as $p\in (3+\frac{4}{N}, \frac{3N+2}{N-2})$ if $N \geq 3$ and $p\in (3 + \frac{4}{N}, \infty)$ if $N=1,2$;
       \item [iii)] For $p=3+ \frac{4}{N}$, there exists a $c_N>0$, given by
      $$c_N : = \inf \{c>0:\ \exists\ u\in S(c) \mbox{ such that } J(u)\leq 0\},$$
      such that
      \begin{equation*}
      \left\{
      \begin{array}{l}
      m(c)= 0,\ \quad \mbox{ as } c \in (0,\ c_N);\\
      m(c)=-\infty,\ \mbox{ as } c \in (c_N, \infty).
      \end{array} \right.
      \end{equation*}
  \end{itemize}
  \item [(2)] When $p\in (1, 1+\frac{4}{N})$, for all $c>0$, $m(c)<0$ and $m(c)$ has a minimizer.
  \item [(3)] When $p\in (1+\frac{4}{N}, 3+\frac{4}{N})$, there exists a $c(p,N)>0$, given by
 \begin{eqnarray}\label{cpn}
 c(p,N):=\inf \{c>0:\ m(c)<0\},
 \end{eqnarray}
 such that
  \begin{itemize}
    \item [i)] If $c\in (0, c(p,N))$, $m(c)=0$ and $m(c)$ has no minimizer;
    \item [ii)] If $c=c(p,N)$, $m(c)=0$ and $m(c)$ admits a minimizer;
    \item [iii)] If $c\in (c(p,N), \infty)$, $m(c)<0$ and $m(c)$ admits a minimizer.
  \end{itemize}
  \item[(4)] When $p = 3 + \frac{4}{N}$, for all $c>0$, $m(c)$ admits no minimizers.
  \item [(5)] The standing waves obtained as minimizers of $m(c)$ are orbitally stable.
\end{itemize}
\end{lem}
In this paper, we mainly deal with the case $p\in (1+\frac{4}{N}, 3+\frac{4}{N})$, $N\geq 1$. From Lemma \ref{lm0} (3), we know that in this range the functional $J$ has, for $c \geq c(p,N)$, a critical point on $S(c)$, which is a global minimizer. Here we extend this result in two directions. First we prove that there exists a $c_0 \in (0, c(p,N))$ such that, for each $c \in (c_0, c(p,N))$ the functional $J$ admits on $S(c)$ a local minimizer. From Lemma \ref{lm0} (3) i) this local minimizer is not a global one. Secondly we show that when $c \in (c_0, \infty)$ the functional $J$ admits on $S(c)$ a second critical point of mountain pass type. Note that since $J$ is not differentiable we must give a meaning to what we call a critical point of $J$ on $S(c)$. By definition it will be a solution of \eqref{eq1.2}, for some $\lambda \in \R$, belonging to $S(c)$. \medskip

%

The main result of this paper is the following theorem.


\begin{thm}\label{mainresult1}
Assume that $p\in (1+\frac{4}{N}, 3+\frac{4}{N})$ with $N \geq 1$.

i) There exists a $c_0 \in (0, c(p,N))$ such that for any  $c \in (c_0, \infty)$ the functional $J$ admits a critical point $v_c$ on $S(c)$ which is a local minimum of $J$ when $c \in (c_0, c(p,N))$ and a global minimum of $J$ when $c \in [c(p,N), \infty)$. In particular,
\begin{center}
\ $J(v_c)\left\{\begin{matrix}
\ >0,\ \mbox{ if }\ c \in (c_0, c(p,N));\\
\ =0, \ \mbox{ if }\ c = c(p,N);\qquad \\
\ <0,\ \mbox{ if }\ c \in (c(p,N), \infty).
\end{matrix}\right.$
\end{center}
\vskip1mm
ii) Assuming in addition that $p\in (1+\frac{4}{N}, \frac{N+2}{N-2}]$ if $N \geq 5$ there exists a second critical point $u_c \in S(c)$ which satisfies
\begin{enumerate}
\renewcommand{\labelenumi}{(\theenumi)}
\item $J(u_c) >0$ for all $c \in (c_0, \infty)$ and $J(u_c)$ is a mountain pass level.
\vskip1mm
\item $J(u_c) > J(v_c)$ for any $c \in (c_0, \infty)$.
\end{enumerate}
For the two critical points we have
\begin{itemize}
\item [(3)] $u_c$ and $v_c$ are Schwarz symmetric functions.
\item [(4)] There exist  Lagrange multipliers $\lambda_c <0$  and $\beta_c <0$ such that
$(u_c, \lambda_c)$ and $(v_c, \beta_c)$ solve weakly \eqref{eq1.2}.
\end{itemize}
\end{thm}

\begin{remark}\label{conjecture}
In \cite[Theorem 1.6]{JL2} it is proved that there exists a $\hat c >0$ such that for all $c \in (0, \hat c)$ the functional $J$ restricted to $S(c)$ has no critical point. It is an open question whether or not we can take $c_0 = \hat c $ in Theorem \ref{mainresult1}. Already it would be interesting to know if the set of $c \in (0, c(p,N)]$ where one can find the two critical points $u_c$ and $v_c$  is an interval.
\end{remark}



In order to obtain multiple critical points we need to use minimax and deformation type arguments, which are more or less standard for smooth variational formulations.
However, the variational functional $J$ here is not well defined in the most natural functions space $W^{1,2}(\R^N)$ let alone being smooth. If a small subspace of $W^{1,2}(\R^N)$ is used then one loses compactness.
To overcome this difficulty of the lack of differentiability of $J$, we apply a perturbation method recently developed in \cite{LLW}. That is, we consider first the perturbed functional
\begin{eqnarray}\label{func2}
J_{\mu}(u):=\frac{\mu}{4}\int_{\R^N}|\nabla u|^4dx + J(u),
\end{eqnarray}
where $\mu\in (0,1]$ is a parameter. For any given $c>0$, we denote
$$\Sigma_c:=\Big \{u\in W^{1,4}\cap W^{1,2}(\R^N):\ \int_{\R^N}|u|^2dx=c \Big\}.$$
One may observe that $J_{\mu}(u)$ is well-defined and $C^1$ in $\Sigma_c$ (see \cite{LLW}).
\vskip2mm

The idea is to look for critical points of $J_{\mu}$, for $\mu >0$ small by using minimax and deformation arguments first. This can be accomplished by rather standard applications of existing techniques. Then, having obtained these critical points for the perturbed problems, we consider convergence of these critical points as $\mu \to 0$. Here we will employe the techniques developed in \cite{LLW} to obtain a certain strong convergence to critical points of the original functional
$J$. \medskip

A first critical point $u_{\mu}^c$ of $J_{\mu}$ is obtained at a level $\gamma_{\mu}(c) >0$ which corresponds to a mountain pass. When $c \in (c_0, c(p,N))$ a second critical point  $v_{\mu}^c$ is obtained as a local minimizer of $J_{\mu}$. The corresponding energy level $\tilde{m}_{\mu}(c)$ is strictly positive. To derive these results we first establish geometric properties of $J_{\mu}$ allowing to search for such critical points. To show that these critical levels are actually reached, several difficulties have to be overcome. Since $J_{\mu}$ is coercive on $\Sigma_c$  any Palais-Smale sequence $\{w_n\} \subset \Sigma_c$ is bounded and thus we can assume that $w_n \rightharpoonup w_c$. It is also standard to show that there exists a $\lambda_{\mu} \in \R$ such that $J'_{\mu}(w_c) - \lambda_{\mu} w_c = 0.$ Finally we can assume that $\{w_n\} \subset \Sigma_c$ satisfies $(w_n - v_n) \to 0$ in $W^{1,4} \cap W^{1,2}(\R^N)$ where $\{v_n\} \subset \Sigma_c$ is a sequence of Schwarz symmetric functions. We shall then say that $\{w_n\} \subset \Sigma_c$ is a sequence of almost Schwarz symmetric functions. This property implies that $w_n \to w_c$ strongly in $L^q(\R^N)$ for $q\in (2, 2 \cdot 2^*)$. The first main difficulty is to show that $w_c \neq 0$. To overcome it we establish, both for  $\gamma_{\mu}(c)$ and $\tilde{m}_{\mu}(c)$, the existence of a Palais-Smale sequence having the additional property that  $Q_{\mu}(w_n) \to 0$. Here $Q_{\mu}(u)$ is defined by
$$Q_{\mu}(u):= \frac{d}{dt}J_{\mu}(u^t)\ |_{t=1},$$
where for any $u \in \Sigma_c$,  $u^t(x):=t^{\frac{N}{2}}u(tx), \, t>0, \, x\in \R^N$. Note that $u^t \in \Sigma_c$ if $u \in \Sigma_c$. Actually the condition that $Q_{\mu}(u)=0$ corresponds to the Pohozaev identity tranferred on the $L^2$-sphere, see \eqref{defQ} for its precise form. \medskip

In the case of $\gamma_{\mu}(c)$ its existence is proved introducing an auxiliary functional on $\Sigma_c \times \R$. This trick of adding one dimension of space, first presented in \cite{LJ} and used recently on various problems \cite{ADP2,BV, HiIkTa, Luo, MoSc}, permits to incorporate into the variational procedure the information that any critical point of $J_{\mu}$ on $\Sigma_c$ must satisfy $Q_{\mu}(u)=0$. For $\tilde{m}_{\mu}(c)$ we can {\it directly} construct a minimizing sequence $\{u_n\} \subset \Sigma_c$ satisfying $Q_{\mu}(u_n) =0, \forall n \in \N$. The existence of a Palais-Smale sequence whose weak limit is non trivial follows. \medskip

Another difficulty is to show that the weak limit $w_c$ belongs to $\Sigma_c$, namely that $||w_c||_{L^2(\R^N)}^2 = c$. For this we need to insure that  $\lambda_{\mu} <0$. In the case of $\tilde{m}_{\mu}(c)$  we show that this is true for any $p \in (1 + \frac{4}{N}, 3 + \frac{4}{N}), N\geq 1$, provided that $c_0$ is sufficiently close to $c(p,N)$. For $\gamma_{\mu}(c) $ we need to require in addition that $ p \in (1 + \frac{4}{N},  \frac{N+2}{N-2}]$ when $N \geq 5.$  This restriction is linked to the fact that  $(P_{\lambda})$ does not have a positive radially, symmetric decreasing solution when $\lambda  \geq 0$ if $p \in (1 + \frac{4}{N}, 3 + \frac{4}{N})$ for $N \leq 4$ or $ p \in (1 + \frac{4}{N},  \frac{N+2}{N-2}]$ when $N \geq 5.$ On the contrary it has such a solution when $\lambda =0$ for $p \in (\frac{N+2}{N-2}, 3 + \frac{4}{N})$ and $N \geq 5$. It is unclear to us if such non existence result is necessary in our problem. Let us mention that we faced in \cite{BJL} a similar issue. There it appears that the presence of a solution corresponding to $\lambda =0$ for a problem related to $(P_{\lambda})$ was an obstacle for the weak limit to belong to the constraint. \medskip

Having proved the existence of the critical points $u_{\mu}^c$ and $v_{\mu}^c$ at the levels $\gamma_{\mu}(c)$ and $\tilde{m}_{\mu}(c)$ respectively we pass to the limit $\mu \to 0$ and we show that $u_{\mu}^c \to u_c$ and $v_{\mu}^c \to v_c$ where $u_c$ and $v_c$ are as presented in Theorem \ref{mainresult1}. At this step we rely on the approach developed in \cite{LLW}  by adopting some arguments though an unconstrained quasi-linear problem was considered there.
\medskip

\vskip2mm

In this paper, we also discuss the behavior of the Lagrange multipliers corresponding to the global minimizers of $J$.
\begin{lem}\label{asymLagrange}
Assume that $p \in (1, 3+ \frac{4}{N})$, $N \geq 1$ and let $v_c$ be a global minimum of $J$ on $S(c)$ and $\beta_c <0$ be its Lagrange multiplier. Then
$\beta_c \to - \infty \mbox{ as } c \to \infty.$
\end{lem}

 Finally we explicit a  relationship between the ground states of \eqref{eq1.2} and the global minimizers of $m(c)$.
\begin{thm}\label{th1.2}
Assume that $p \in (1, 3 + \frac{4}{N})$, $N \geq 1$  and let $v_c$ be a global minimum of $J$ on $ S(c)$ and $\beta_c<0$ be its Lagrange multiplier. Then $v_c$ is a ground state solution of \eqref{eq1.2} with $\lambda= \beta_c$.
\end{thm}

In Remark \ref{RNC}, we show that the converse of Theorem \ref{th1.2} does not hold.  \\

We end this introduction by pointing out that whether or not the standing waves associated to our critical points are orbitally stable is an open question. We conjecture that our mountain pass solution leads to orbitally unstable standing wave and that a local minimizer leads to a stable one. \\


\textbf{Acknowledgements.} The first author thanks N. Boussaid and T. Cazenave for useful discussions. The authors also express their sincere gratitude to T. Watanabe for providing to them, in addition to  useful references, the complete statement and proof of Lemma \ref{lm-decay}. This work has been carried out in the framework of the project NONLOCAL (ANR-14-CE25-0013), funded by the French National Research Agency (ANR). The third author is supported by NSFC-11271201 and BCMIIS. The authors also would like to thank the referee for his/her valuable comments and suggestions.\\


\subsection{Notations}
In the paper, for any $1\leq s < \infty$, $L^s(\R^N)$ is the usual Lebesgue space endowed
with the norm
$$\|u\|_{s}^s:=\int_{\R^{N}} |u|^sdx,$$
and $W^{1,s}(\R^N)$ the usual Sobolev space  endowed with the norm
$$\|u\|_{W^{1,s}}:=\|\nabla u\|_s + \|u\|_s.$$
We denote by $X$ the space $W^{1,4}\cap W^{1,2}(\R^N)$ equipped with its natural norm $||\cdot||_X := ||\cdot||_{W^{1,4}} + ||\cdot||_{W^{1,2}}$.
Throughout the paper we shall denote by $C>0$ various positive constants which may vary from one line to another but which do not affect the analysis of the problem.

\section{Non-existence results}

In this section, we prove two non existence results for solutions to  \eqref{eq1.2}. These results which have their own interest will be crucially used to control the possible values of our Lagrange parameters. We also discuss the restriction we have to impose when $N \geq 5$.

\begin{lem}\label{lem2.1}
Assume that $p \in (1, \frac{N+2}{N-2}]$ if $N\geq 3$, $p \in (1, \infty)$ if $N=1,2$, and that $\lambda\geq 0$. Then \eqref{eq1.2}
has only trivial solutions in $\mathcal{X}$.
\end{lem}
\begin{proof}
Let $(u, \lambda)$ solves \eqref{eq1.2}. Then it can be proved that $(u, \lambda)$ satisfies the following Pohozaev identity
\begin{equation}\label{poho2.1}
\frac{N-2}{N} \left(\frac{1}{2}\|\nabla u\|_2^2+\int_{\R^N}|u|^2|\nabla u|^2dx  \right) -\frac{\lambda}{2}\|u\|_2^2 = \frac{1}{p+1}\|u\|_{p+1}^{p+1}.
\end{equation}
see \cite[Lemma 3.1]{CJS}. In addition, testing  \eqref{eq1.2} by $u$, we have
\begin{equation}\label{poho2.2}
||\nabla u||_2^2 + 4 \int_{\R^N}|u|^2 |\nabla u|^2 dx - \lambda ||u||_2^2 - ||u||_{p+1}^{p+1}=0.
\end{equation}
By a direct calculation, it follows from \eqref{poho2.1} and \eqref{poho2.2} that
\begin{equation}\label{3}
\lambda \|u\|_2^2 = \frac{(N-2)p-(N+2)}{N(p-1)}\|\nabla u\|_2^2 + 2\frac{(N-2)p-(3N+2)}{N(p-1)}\int_{\R^N} |u|^2 |\nabla u|^2dx.
\end{equation}
Under our assumptions, (\ref{3}) implies that if $\lambda  \geq 0$ necessarily $u=0$.
\end{proof}

To prove our second non existence result we need the following Lemma:
\begin{lem}(\cite[Lemma A.2]{Ik})\label{ikoma}
Suppose $r \in (1, \frac{N}{N-2}]$, $N \geq 3$ and let $v \in L^r(\R^N)$ be a smooth non-negative function which satisfies $- \Delta v \geq 0$ in $\R^N$. Then $v \equiv 0$ holds.
\end{lem}

Using this Liouville's result we get

\begin{lem}\label{lem2.22}
Assume that $p \in (\frac{N+2}{N-2}, \frac{3N+2}{N-2})$ with $N =3, 4$, and that $\lambda\geq 0$. Then \eqref{eq1.2} has no non-negative solution.
\end{lem}

\begin{proof}
First observe that (e.g. \cite{CJ, LWW1}) the search of non-negative solutions of \eqref{eq1.2} is equivalent to the search of non-negative solutions of
\begin{equation}\label{semi-eq}
- \Delta v = \frac{1}{\sqrt{1+2f^2(v)}} \Big(|f(v)|^{p-1}f(v) + \lambda f(v) \Big), \quad v \in H^1(\R^N),
\tag{$E_{\lambda}$}
\end{equation}
where $f$ is the unique solution of the Cauchy problem
\begin{eqnarray*}
\left\{
\begin{array}{l}
f'(s)=\frac{1}{\sqrt{1+2f^2(s)}},\\
f(0)=0
\end{array}
\right.
\end{eqnarray*}
on $[0, + \infty[$ and $f(t) = - f(-t)$ on $] - \infty, 0].$
Indeed, by \cite{CJ} it is known that any solution $v\in H^1(\R^N)$ of \eqref{semi-eq} gives arise to a solution $u:=f(v)$ of \eqref{eq1.2}. Reciprocally, from \cite[Lemmas 2.6, 2.8, 2.10]{AT} we know that for any solution $u$ of \eqref{eq1.2}, $v:=f^{-1}(u)$ is a solution of \eqref{semi-eq}. Clearly also a non-negative solution of  \eqref{eq1.2} is transformed into a non-negative solution of \eqref{semi-eq}. Now since \eqref{semi-eq} implies that
$ - \Delta v \geq 0,$ in $ \R^N,$
the conclusion of the lemma follows directly from Lemma \ref{ikoma}.
\end{proof}

Before ending this section and to enlighten the restriction we have to impose when $N \geq 5$ note that

\begin{lem}\label{lem2}
Assume that $p \in (\frac{N+2}{N-2}, \frac{3N+2}{N-2})$, $N  \geq 5$, and that $\lambda=0$. Then \eqref{eq1.2} has positive non trivial radially symmetric solution in $L^2(\R^N)$.
\end{lem}

Let us observe that a positive radially symmetric decreasing solution of \eqref{semi-eq} with $\lambda =0$ exists in $\mathcal{D}^{1,2}(\R^N)$ if $p> \frac{N+2}{N-2}$. This result is given in \cite[Proposition 4.1]{ASW}. To prove the lemma we need to show that, under our assumptions, it does belong to $L^2(\R^N)$. Indeed, since $f(s)/s \to 1$ as $s \to 0$ a solution of \eqref{eq1.2} belongs to $L^2(\R^N)$ if and only if the corresponding solution of \eqref{semi-eq} also belongs to $L^2(\R^N)$. Also a positive radially symmetric decreasing solution of \eqref{eq1.2} corresponds to a positive radially symmetric decreasing solution of \eqref{semi-eq}. The following result whose proof will be postponed to the Appendix is due to T. Watanabe \cite{Wa}. It provides the decay estimates we need to prove Lemma \ref{lem2}.

\begin{lem}\label{lm-decay}
Assume that $p\in (\frac{N+2}{N-2}, \frac{3N+2}{N-2}), N\geq 3$.
Let $v_0(r)\in \mathcal{D}^{1,2}(\R^N)$ be a $C^2$, positive and radially symmetric decreasing solution of \eqref{semi-eq}
with $\lambda=0$, namely
\begin{eqnarray}\label{semi-eq-0}
- \Delta v_0 = f(v_0)^{p}f'(v_0), \mbox{ in } \R^N.
\end{eqnarray}
Then there exists a positive constant $C_{\infty}>0$ such that for sufficiently large $r>0$, $v_0$ satisfies
\begin{eqnarray}\label{decayp2}
C_{\infty}K(r)\Big(1-O(r^{-2})\Big)\leq v_0(r)\leq C_{\infty}K(r).
\end{eqnarray}
\begin{eqnarray}\label{decayp3}
C_{\infty}K'(r) \leq v_0'(r)\leq C_{\infty}K'(r)\Big(1-O(r^{-2})\Big).
\end{eqnarray}
Here $$K(r):=\frac{1}{(N-2)|S^{N-1}|r^{N-2}}, C>0,$$
is the fundamental solution of $-\Delta$ on $\R^N$.
\end{lem}

\begin{proof}[Proof of Lemma \ref{lem2}]
For a solution $v \in \mathcal{D}^{1,2}(\R^N)$ to belong to $L^2(\R^N)$ the decay of $v$ at infinity must be such that
\begin{eqnarray}\label{L2}
\int_0^{\infty}r^{N-1}v^2(r)dr<\infty.
\end{eqnarray}
Note from \eqref{decayp2} in Lemma \ref{lm-decay} that
\begin{eqnarray*}
\int_0^{\infty}r^{N-1}v^2(r)dr<\infty &\Longleftrightarrow& \int_0^{\infty}r^{N-1}\frac{1}{r^{2N-4}}dr<\infty \\
&\Longleftrightarrow& \int_0^{\infty}\frac{1}{r^{N-3}}dr<\infty.
\end{eqnarray*}
Then \eqref{L2} is true if and only if $N-3>1$, namely if and only if $N\geq 5$. This ends the proof.
\end{proof}

\section{Perturbation of the functional}
In this section, to overcome the non-differentiability of the functional $J$, we apply the perturbation approach introduced in \cite{LLW}. First, we show that there exists a $c_0\in (0, c(p,N))$ such that, for each $c\in (c_0, \infty)$ the functional $J_{\mu}$ has a mountain pass geometry on $\Sigma_c$ when $\mu>0$ is sufficiently small.

\begin{lem}\label{lem1}[Mountain Pass geometry]
Assume that $p\in (1+\frac{4}{N}, 3+\frac{4}{N})$, $N \geq 1$. Then there exists a $c_0 \in (0, c(p,N))$ such that for any fixed $c \in [c_0, \infty)$ taking $\mu_0 >0$ small enough the functional $J_{\mu}$ has, for $\mu \in (0, \mu_0)$ a Mountain Pass geometry on the constraint $\Sigma_c$. Precisely there exist  $(u_0, u_1)\in \Sigma_c \times \Sigma_c$  both Schwarz symmetric, such that
\begin{eqnarray}\label{mp}
\gamma_{\mu}(c)=\inf_{g\in \Gamma_c}\max_{t\in [0,1]}J_{\mu}(g(t))> \max \{J_{\mu}(u_0), J_{\mu}(u_1)\}
\end{eqnarray}
where
$$\Gamma_c=\{g\in C([0,1], \Sigma_c):\ g(0)=u_0 ,\ g(1)=u_1 \}.$$
\end{lem}
The proof of Lemma \ref{lem1} relies on the following estimates which holds under more general assumptions than those of Theorem \ref{mainresult1}.
\vskip1mm

\begin{lem}\label{estimates}
Assume that $p\in (1+\frac{4}{N},  \frac{3N+2}{N-2})$ if $N\geq 3$ and $p\in (1+\frac{4}{N}, \infty)$ if $N=1,2$. Then setting for $k >0$,
$$ C_k:=\Big\{u \in \Sigma_c: \int_{\R^N}(1+|u|^2)|\nabla u|^2dx = k\Big\}$$
there exists a $k_0 >0$ sufficiently small such that for all $k \in (0, k_0]$ and all $\mu >0$
\begin{equation}\label{lJ}
J_{\mu}(u)\geq \frac{1}{4}k>0,\ \mbox{ for all } u\in C_{k};
\end{equation}
\begin{equation}\label{lQ}
Q_{\mu}(u)\geq \frac{1}{4}k>0,\ \mbox{ for all } u\in C_{k}.
\end{equation}
Here we have set, for any given $\mu >0$,
\begin{eqnarray}\label{defQ}
Q_\mu(u) &:=& \frac{\mu (N+4)}{4}||\nabla u||_4^4 + ||\nabla u||_2^2 + (N+2) \int_{\R^N}|u|^2 |\nabla u|^2 dx  \nonumber \\
&-& \frac{N(p-1)}{2(p+1)}||u||^{p+1}_{p+1}.
\end{eqnarray}
Moreover when $c \in (0, c(p,N)]$, the constant $k_0 >0$ can be chosen independently of $c >0$.
\end{lem}

\begin{proof}
When $p\in (1+\frac{4}{N}, \frac{N+2}{N-2})$ if $N\geq 3$ and $p\in (1+\frac{4}{N}, \infty)$ if $N=1,2$, by the Gagliardo-Nirenberg inequality, there exists a $K_1=K_1(p,N)>0$, such that for all $u\in X$,
\begin{eqnarray}\label{11}
\int_{\R^N}|u|^{p+1}dx &\leq& K_1\left (\int_{\R^N}  |\nabla u|^2 dx\right )^{\frac{N(p-1)}{4}}\left ( \int_{\R^N} |u|^2dx \right )^{\frac{(N+2)-(N-2)p}{4}}\nonumber\\
&\leq& K_1\left (\int_{\R^N} (1+|u|^2) |\nabla u|^2 dx\right )^{\frac{N(p-1)}{4}}\left ( \int_{\R^N} |u|^2dx \right )^{\frac{(N+2)-(N-2)p}{4}}.
\end{eqnarray}
Thus, for a constant $K_2 = K_2(p, N) >0$ independent of $c>0$ when $c \in (0, c(p,N))$, we have for all $u \in \Sigma_c$ that
\begin{equation}\label{13}
J_{\mu}(u) \geq \frac{1}{2}\int_{\R^N}(1+ |u|^2 ) |\nabla u|^2 dx - K_2\left (\int_{\R^N} (1+|u|^2) |\nabla u|^2 dx\right )^{\frac{N(p-1)}{4}}.
\end{equation}
\vskip3mm

When $p\in [ \frac{N+2}{N-2}, \frac{3N+2}{N-2}), N\geq 3$, we claim that there exists a $K_3=K_3(p,N)>0$, such that for all $u\in X$,
\begin{eqnarray}\label{12}
\int_{\R^N}|u|^{p+1}dx \leq K_3\left (\int_{\R^N}(1+ |u|^2 ) |\nabla u|^2 dx\right )^{\frac{N}{N-2}}.
\end{eqnarray}
To show \eqref{12}, let $\alpha\in \R$ be such that
$$\frac{p+1}{\alpha+1}=\frac{2N}{N-2}.$$
 Clearly $\alpha\in [0,1)$ if $p\in [ \frac{N+2}{N-2}, \frac{3N+2}{N-2}), N\geq 3$. Then by the Gagliardo-Nirenberg-Sobolev inequality, there exists a $K_3=K_3(p,N)>0$ such that for all $u\in X$,
\begin{eqnarray*}
\int_{\R^N}|u|^{p+1}dx &=& \int_{\R^N}(|u|^{\alpha+1})^{\frac{2N}{N-2}}dx \leq K_3\left(\int_{\R^N}|\nabla(|u|^{\alpha+1})|^2dx \right)^{\frac{N}{N-2}}\\
&=& K_3\left(\int_{\R^N}|u|^{2\alpha}|\nabla u|^2dx \right)^{\frac{N}{N-2}}.
\end{eqnarray*}
Since $\alpha \in [0,1)$, from the last inequality we then obtain \eqref{12}. Thus when $p\in [ \frac{N+2}{N-2}, \frac{3N+2}{N-2}), N\geq 3$, for any $u\in X$, one has
\begin{equation}\label{14}
J_{\mu}(u) \geq \frac{1}{2}\int_{\R^N}(1+ |u|^2 ) |\nabla u|^2 dx - \frac{K_3}{p+1}\left (\int_{\R^N}(1+ |u|^2 ) |\nabla u|^2 dx\right )^{\frac{N}{N-2}},
\end{equation}
\vskip2mm

Note that $\frac{N(p-1)}{4} >1$ as $p>1+\frac{4}{N}$ and $\frac{N}{N-2}>1$ as $N\geq 3$.  Thus when $p$ satisfies the assumption of the lemma, by \eqref{13}-\eqref{14}, we conclude that there exists a $k_0>0$ small, such that for all $k \in (0, k_0)$,
$$
J_{\mu}(u)\geq \frac{1}{4}k>0,\ \mbox{ for all } u\in C_{k}.$$
In addition, we note that when $c \in (0, c(p,N)]$ the constant $k_0 >0$ can be chosen independently of $c >0$. This proves that (\ref{lJ}) holds. Now
by the estimates \eqref{13} and \eqref{14}, one obtains the existence of a constant $K_4=K_4(p,N)>0$, independent of $u\in X$, such that for all $u\in \Sigma_c$,
\begin{eqnarray}\label{esti11}
Q_{\mu}(u)\geq \int_{\R^N}(1+ |u|^2 ) |\nabla u|^2 dx - K_4 \Big \{ \int_{\R^N}(1+ |u|^2 )|\nabla u|^2 dx  \Big \}^{\frac{N(p-1)}{4}},
\end{eqnarray}
where $p\in (1+ \frac{4}{N}, \infty)$ if $N=1,2$ and $p\in (1+ \frac{4}{N}, \frac{N+2}{N-2})$ if $N\geq 3$, and
\begin{eqnarray}\label{esti22}
Q_{\mu}(u)\geq \int_{\R^N}(1+ |u|^2 ) |\nabla u|^2 dx - K_4 \Big \{ \int_{\R^N}(1+ |u|^2 ) |\nabla u|^2 dx  \Big \}^{\frac{N}{N-2}},
\end{eqnarray}
where $p\in [\frac{N+2}{N-2}, \frac{3N+2}{N-2})$ as $N\geq 3$. From \eqref{esti11}-\eqref{esti22} we also  derive that (\ref{lQ}) holds.
\end{proof}

\begin{proof}[Proof of Lemma \ref{lem1}]

To choose $u_1$ we distinguish the cases $c > c(p,N)$ and $c \leq c(p,N)$.  When $c > c(p, N)$ we know from Lemma \ref{lm0} (3) that $J$ has a global minimizer $u_c \in W^{1,2}(\R^N)$ satisfying $J(u_c) = m(c) <0$. Without restriction we can assume that $u_c$ is Schwarz symmetric and combining  \cite[Lemma 4.6]{CJS} and  \cite[Lemma 5.10]{LWW2}, we see that $u_c$, having an exponential decrease at infinity, belongs to $X$. Thus setting $u_1 = u_c$ and taking $k_0>0$ smaller if necessary we then have
\begin{equation}\label{star}
J(u_1) <0 \quad \mbox{ and } \quad \int_{\R^N} (1+|u_1|^2) |\nabla u_1|^2 dx >k_0.
\end{equation}
The value $k_0 >0$ being defined in Lemma \ref{estimates}. Now taking $\mu_0 >0$ small enough we have by continuity that $J_{\mu}(u_1) <0$ for all $\mu \in (0, \mu_0)$. \medskip

When $c \leq c(p,N)$ we use the fact that for $c = c(p,N)$, $J$ has a global minimizer $u_{c(p,N)}$ with $J(u_{c(p,N)})= 0$ that we can assume to be Schwarz symmetric and in $X$. We set
$$ k_1 = \int_{\R^N}(1 + |u_{c(p,N)}|^2) |\nabla u_{c(p,N)}|^2 dx.$$
Restricting $k_0 >0$ in Lemma \ref{estimates} if necessary we can assume that $2k_0 <  k_1$. By continuity there exists a $t_0 <1$ such that, for any $t \in (t_0,1)$
$$J(\sqrt{t} u_{c(p,N)}) < \frac{1}{4}k_0$$
and
$$\int_{\R^N}(1 + | \sqrt{t}u_{c(p,N)}|^2 ) |\nabla (\sqrt{t} u_{c(p,N)})|^2 dx \geq \frac{3}{2}k_0.$$
Note that $\sqrt{t} u_{c(p,N)} \in \Sigma_c$ with $c = t c(p,N)$. \medskip

Now we set $c_0 = t_0 c(p,N)$ and for each $c \in (c_0, c(p,N)),$ we choose $u_1$ as $u_1 = \sqrt{t} u_{c(p,N)}$ where $t \in (t_0,1)$ is such that $u_1 \in \Sigma_c$. Finally taking $\mu_0 >0$ small enough we have by continuity that $J_{\mu}(u_1) < \frac{1}{4}k_0$ for all $\mu \in (0, \mu_0)$.



\medskip

To choose $u_0\in \Sigma_c$, we consider, for any arbitrary Schwartz symmetric function, the scaling
$$v^{\theta}(x):= \theta^{N/2}v(\theta x),\  \forall\ \theta>0.$$
By direct calculations we observe that
$$ v^{\theta}\in \Sigma_c, \forall\ \theta>0, \quad \lim_{\theta\to 0}J_{\mu}(v^{\theta})=0 \quad \mbox{and} \quad \lim_{\theta\to 0}\int_{\R^N} (1+|v^{\theta}|^2) |\nabla v^{\theta}|^2 dx =0.$$
Thus,  setting $u_0 = v^{\theta_0}$ for a fixed $\theta_0>0$ sufficiently small we have
$$J_{\mu}(u_0)\leq \frac{1}{8}k_0, \quad \mbox{ and } \quad \int_{\R^N} (1+|u_0|^2) |\nabla u_0|^2 dx<k_0.$$
\end{proof}

We now show that the geometry of $J$  presents a local minimum when $c \in (c_0, c(p,N))$. Actually we shall get a local minimizer of $J_{\mu}$ on $\Sigma_c$  by considering the minimization problem
\begin{equation}\label{minimum}
\tilde{m}_{\mu}(c):=\inf_{u\in \Sigma_c \setminus B} J_{\mu}(u),
\end{equation}
where $B:= \bigcup_{0<k\leq k_0}C_k$ and $k_0>0$ is given in Lemma \ref{estimates}.
\begin{lem}[Geometry of local minima]\label{defmin}
For any given $c \in (c_0, c(p,N))$ we have
\begin{equation}\label{minimum2}
0 \leq \tilde{m}_{\mu}(c)=\inf\limits_{u\in \Sigma_c \setminus B} J_{\mu}(u)< \inf_{u\in C_{k_0}}J_{\mu}(u).
\end{equation}
\end{lem}

\begin{proof}
From the proof of Lemma \ref{lem1}, we know that there exists a $v_0\in \Sigma_c \setminus B$ such that $J_{\mu}(v_0)<\inf\limits_{u\in C_{k_0}}J_{\mu}(u)$. Thus
$$\inf_{u\in \Sigma_c \setminus B} J_{\mu}(u)\leq J_{\mu}(v_0)<\inf_{u\in C_{k_0}}J_{\mu}(u).$$
In addition, by Lemma \ref{lm0}, we know that $\tilde{m}_{\mu}(c)=\inf_{u\in \Sigma_c \setminus B} J_{\mu}(u)\geq 0$. Thus the proof is completed.
\end{proof}

In view of Lemmas \ref{lem1} and \ref{defmin} we shall search for critical points of $J_{\mu}$ at the levels $\gamma_{\mu}(c)$ and $\tilde{m}_{\mu}(c)$. For this we establish the existence of special Palais-Smale sequences associated with $\gamma_{\mu}(c)$ and $\tilde{m}_{\mu}(c)$ and then show their convergence. In that direction we first observe

\begin{lem}\label{boundedness}
Assume that $p\in (1+\frac{4}{N}, 3+\frac{4}{N})$, $N \geq 1$. Then for any fixed $c \in (0, \infty)$ and any fixed $\mu >0$, if a sequence $\{u_n\} \subset \Sigma_c$ is such that $\{J_{\mu}(u_n)\} \subset \R$ is bounded then it is bounded in $X$.
\end{lem}

\begin{proof}
By the Gagilardo-Nirenberg-Sobolev inequality, it has been proved (see \cite[(4.5)]{CJS}) that, for any $u\in X$ there holds
\begin{eqnarray}\label{pf4.5.1}
\int_{\R^N}|u|^{q+1}dx \leq C \cdot \|u\|_{2}^{\frac{3N+2-(N-2)q}{N+2}}\cdot \left( \int_{\R^N}|u|^2|\nabla u|^2dx  \right)^{\frac{N(q-1)}{2(N+2)}},
\end{eqnarray}
for $q\in (1, \frac{3N+2}{N-2})$ if $N\geq 3$ and $q\in (1, \infty)$ if $N=1,2$, $C>0$ is a constant, independent of $u$. Thus we have
\begin{eqnarray*}
J_{\mu}(u_n)\geq \frac{\mu}{4}\|\nabla u_n\|_4^4 + \frac{1}{2}\int_{\R^N}(1+|u_n|^2)|\nabla u_n|^2dx-C\cdot \left( \int_{\R^N}(1+|u_n|^2)|\nabla u_n|^2dx   \right)^{\frac{N(p-1)}{2(N+2)}}.
\end{eqnarray*}

When $p<3+\frac{4}{N}$, we have $\frac{N(p-1)}{2(N+2)}<1$. Then the last inequality implies from the boundedness of $\{J_{\mu}(u_n)\}$ that $\{\int_{\R^N}(1+|u_n|^2)|\nabla u_n|^2dx\}$ is bounded, and also for fixed $\mu>0$, $\{\|\nabla u_n\|_4^4\}$ is bounded. In addition, because of \eqref{pf4.5.1}, $\{\|u_n\|_4^4\}$ is bounded. Thus $\{u_n\}$ is bounded in $X$.
\end{proof}

From Lemma \ref{boundedness} we know, in particular, that any Palais-Smale sequence for $J_{\mu}$ is bounded. The need to use special Palais-Smale sequences comes from the difficulty to pass from weak convergence to strong convergence. A problem due to the fact that our equation is set on all $\R^N$. In order to regain some compactness we could
proceed as in \cite[Lemma 2.1]{LLW} by working in the subspace of radially symmetry functions $W_r^{1,4}\cap W_r^{1,2}(\R^N)$, $N\geq 2$. However, when $N=1$, the inclusion $H_{r}^{1}(\R) \subset L^{q}(\R)$ for $q>2$ is not compact and another proof is needed. Here we choose to construct  special Palais-Smale sequences for $J_{\mu}$ which consist of almost Schwarz symmetric functions. This allows us to give a unify proof in any dimension.  \medskip

Even if we work with sequences of almost symmetric functions it is not automatic that they converge to a non-trivial weak limit. To avoid this possibility we construct Palais-Smale sequences $\{u_n\} \subset \Sigma_c$ which satisfy in addition the property that $Q_{\mu}(u_n) \to 0$. For the mountain pass level $\gamma_{\mu}(c)$ this is done using the trick introduced in \cite{LJ}. For $\tilde{m}_{\mu}(c)$ a direct argument will provide the result. \medskip

\begin{lem}\label{lm2}[A special Palais-Smale sequence for $\gamma_{\mu}(c)$]
Assume that $p\in (1+\frac{4}{N}, 3+\frac{4}{N})$, $N \geq 1$. Then, for any given $c \geq c_0$ where $c_0 >0$ is given in Lemma \ref{lem1}, there exist a sequence $\{u_n \}\subset \Sigma_c$ and a sequence $\{v_n\}\subset X$ of Schwarz symmetric functions, such that
\begin{equation}\label{SPSC}
\left\{
\begin{array}{l}
J_{\mu}(u_n)\to \gamma_{\mu}(c)>0,\\
\|J'_{\mu}|_{\Sigma_c}(u_n)\|_{X^{\ast}}\to 0,\\
Q_{\mu}(u_n)\to 0,\\
\|u_n-v_n\|_{X}\to 0,\\
\end{array}
\right.
\end{equation}
as $n\to \infty$. Here $X^{\ast}$ denotes the dual space of $X$.
\end{lem}

Before proving Lemma \ref{lm2} we need to introduce some notations and to prove some preliminary results. For any fixed $\mu>0$, we introduce the following auxiliary functional
$$\widetilde{J}_{\mu}: \Sigma_c \times \R \to \R,\qquad  (u, s)\mapsto J_{\mu}(H(u,s)),$$
where $H(u(x),s):= e^{\frac{N}{2}s}u(e^s x)$, and the set of paths
\begin{align*}
\widetilde{\Gamma}_c:= \Big\{\widetilde{g} \in C([0,1], \Sigma_c \times \R) :  \ \widetilde{g}(0)= (u_0, 0),\ \widetilde{g}(1)= (u_1, 0) \Big\},
\end{align*}
where $u_0, u_1 \in \Sigma_c$ are given in Lemma \ref{lem1}. Observe that defining
$$\widetilde{\gamma}_{\mu}(c):= \inf_{\widetilde{g}\in \widetilde{\Gamma}_c}\max\limits_{t\in [0,1]}\widetilde{J}_{\mu}(\widetilde{g}(t)),$$
we have that
\begin{eqnarray}\label{ggamma(c)}
\widetilde{\gamma}_{\mu}(c) = \gamma_{\mu}(c).
\end{eqnarray}
Indeed, by the definitions of $\widetilde{\gamma}_{\mu}(c)$ and $\gamma_{\mu}(c)$, \eqref{ggamma(c)} follows immediately from the fact that the maps
$$\varphi: \Gamma_c \longrightarrow \widetilde{\Gamma}_c,\ g \longmapsto \varphi(g):=(g,0),$$
and
$$\psi: \widetilde{\Gamma}_c \longrightarrow \Gamma_c,\ \widetilde{g} \longmapsto \psi(\widetilde{g}):=H \circ \widetilde{g},$$
satisfy
$$\widetilde{J}_{\mu}(\varphi(g)) = J_{\mu}(g)\ \mbox{ and }\ J_{\mu}(\psi(\widetilde{g})) = \widetilde{J}_{\mu}(\widetilde{g}).$$

In the proof of Lemma \ref{lm2}, the lemma below  which has been established by the Ekeland variational principle in \cite[Lemma 2.3]{LJ} is used. Hereinafter we denote by $E$ the set $ X \times \R$ equipped with the norm $\|\cdot\|_E^2 = \|\cdot\|_{X}^2 + |\cdot|_{\R}^2$ and denote by $E^{\ast}$ its dual space.

\begin{lem}\label{lm-ekeland}
Let $\varepsilon>0$. Suppose that $\widetilde{g}_0 \in \widetilde{\Gamma}_c$ satisfies
$$\max\limits_{t\in [0,1]}\widetilde{J}_{\mu}(\widetilde{g}_0(t))\leq \widetilde{\gamma }_{\mu}(c)+ \varepsilon.$$
Then there exists a pair of $(u_0, s_0)\in \Sigma_c \times \R$ such that:
\begin{itemize}
  \item [(1)] $\widetilde{J}_{\mu}(u_0,s_0) \in [\widetilde{\gamma}_{\mu}(c)- \varepsilon, \widetilde{\gamma}_{\mu}(c) + \varepsilon]$;
  \item [(2)] $\min\limits_{t\in [0,1]} \| (u_0,s_0)-\widetilde{g}_0(t) \|_E \leq \sqrt{\varepsilon}$;
  \item [(3)] $\| \widetilde{J}_{\mu}'|_{\Sigma_c \times \R}(u_0,s_0)  \|_{E^{\ast}}\leq 2\sqrt{\varepsilon}$,\ i.e.
$$|\langle \widetilde{J}_{\mu}'(u_0, s_0), z \rangle_{E^{\ast}\times E}  |\leq 2\sqrt{\varepsilon}\left \| z \right \|_E,$$ holds for all $z \in \widetilde{T}_{(u_0,s_0)}:= \{(z_1,z_2) \in E, \langle u_0, z_1\rangle_{L^2}=0\}$.
\end{itemize}
\end{lem}

\begin{proof}[Proof of Lemma \ref{lm2}] For each $n\in \N$, by the definition of $\gamma_{\mu}(c)$, there exists a $g_n\in \Gamma_c$ such that
$$\max\limits_{t\in [0,1]}J_{\mu}(g_n(t))\leq \gamma_{\mu}(c) + \frac{1}{n}.$$
Denote by $g_n^{\ast}$ the Schwarz symmetrization of $g_n\in \Gamma_c$. Then by the Polya-Szeg\"{o} inequality $\|\nabla u^{\ast}\|_q^q\leq\|\nabla u\|_q^q,\ \forall q\in [1,\infty)$, and using \cite[Lemma 4.3]{CJS}, we have that
$$\max_{t\in [0,1]}J_{\mu}(g_n^{\ast}(t))\leq \max_{t\in [0,1]}J_{\mu}(g_n(t)).$$
Since $\widetilde{\gamma}_{\mu}(c)=\gamma_{\mu}(c)$, then for each $n\in \N$, $\widetilde{g}_n:=(g_n^{\ast}, 0)\in \widetilde{\Gamma}_c$ and satisfies
$$\max\limits_{t\in [0,1]}\widetilde{J}_{\mu}(\widetilde{g}_n(t))\leq \widetilde{\gamma}_{\mu}(c) + \frac{1}{n}.$$

Thus applying Lemma \ref{lm-ekeland}, we obtain a sequence $\{(w_n,s_n)\}\subset \Sigma_c \times \R$ such that:
\begin{itemize}
  \item [(i)] $\widetilde{J}_{\mu}(w_n, s_n) \in [\gamma_{\mu}(c)-\frac{1}{n}, \gamma_{\mu}(c)+\frac{1}{n}]$;
  \item [(ii)] $\min\limits_{t\in [0,1]} \| (w_n, s_n)-(g_n^{\ast}(t), 0) \|_E \leq\frac{1}{\sqrt{n}}$;
  \item [(iii)] $\| \widetilde{J}_{\mu}'|_{\Sigma_c \times \R}(w_n, s_n)  \|_{E^{\ast}}\leq \frac{2}{\sqrt{n}}$,\ i.e.
$$|\langle \widetilde{J}_{\mu}'(w_n, s_n), z \rangle_{E^{\ast}\times E}  |\leq \frac{2}{\sqrt{n}}\left \| z \right \|_E,$$ holds for all $z \in \widetilde{T}_{(w_n, s_n)}:= \{(z_1,z_2) \in E, \langle w_n, z_1\rangle_{L^2}=0\}$.
\end{itemize}
Now we claim that for each $n\in \N$, there exists a $t_n\in [0,1]$ such that $u_n:=H(w_n, s_n)$ and $v_n:= g_n^{\ast}(t_n)$ satisfy \eqref{SPSC}. Indeed, first, from (i) we have that $J_{\mu}(u_n) \to \gamma_{\mu}(c)$, since $J_{\mu}(u_n)=J_{\mu}(H(w_n, s_n))=\widetilde{J}_{\mu}(w_n, s_n)$. Secondly, by simple calculations, we have
$$Q_{\mu}(u_n) = \langle \widetilde{J}_{\mu}'(w_n, s_n), (0,1) \rangle_{E^{\ast}\times E},$$
and $(0,1)\in \widetilde{T}_{(w_n, s_n)}$. Thus (iii) yields that $Q_{\mu}(u_n) \to 0$. To verify that $\|J'_{\mu}|_{\Sigma_c}(u_n)\|_{X^{\ast}}\to 0$, it suffices to prove for $n\in \N$ sufficiently large, that
\begin{eqnarray}\label{F'u}
|\langle J_{\mu}'(u_n), \phi\rangle_{X^{\ast}\times X}  |\leq \frac{4}{\sqrt{n}}\left \| \phi \right \|_X,\ \mbox{ for all }\ \phi \in T_{u_n},
\end{eqnarray}
where $T_{u_n}:=\{\phi \in X,\  \langle u_n,\phi\rangle_{L^2}=0\}$. To this end, we note that, for each $\phi\in T_{u_n}$, by denoting $\widetilde{\phi}=H(\phi, -s_n)$, one has
\begin{eqnarray*}
\langle J'_{\mu}(u_n), \phi \rangle_{X^{\ast}\times X} &=& \mu \int_{\R^N} |\nabla u_n|^2 \nabla u_n \nabla \phi dx  + \int_{\R^N} \nabla u_n \nabla \phi dx  \\
&+& 2\int_{\R^N} ( u_n \phi |\nabla u_n|^2 +  |u_n|^2\nabla u_n \nabla \phi ) dx - \int_{\R^N} |u_n|^{p-1}u_n \phi dx .\\
&=& e^{(N+4)s_n}\cdot \mu \int_{\R^N} |\nabla w_n|^2 \nabla w_n \nabla \phi dx  + e^{2s_n}\int_{\R^N} \nabla w_n \nabla \phi dx  \\
&+& e^{(N+2)s_n}\cdot 2\int_{\R^N} ( w_n \phi |\nabla w_n|^2 +  |w_n|^2\nabla w_n \nabla \phi ) dx \\
&-& e^{\frac{N(p-1)}{2}s_n} \int_{\R^N} |w_n|^{p-1}w_n \phi dx \ =\ \langle \widetilde{J}_{\mu}'(w_n, s_n), (\widetilde{\phi}, 0)\rangle_{E^{\ast}\times E}.
\end{eqnarray*}
If $(\widetilde{\phi},0)\in \widetilde{T}_{(w_n,s_n)}$ and $\|(\widetilde{\phi},0)\|_E^2\leq 4\|\phi\|_X^2$ as $n\in \N$ sufficiently large, then from (iii) we conclude \eqref{F'u}. To verify this condition, one observes that
$(\widetilde{\phi},0)\in \widetilde{T}_{(w_n,s_n)} \Leftrightarrow \phi \in T_{u_n}$, also from (ii) it follows that
\begin{eqnarray}\label{value-s_n}
|s_n|=|s_n-0|\leq \min\limits_{t\in [0,1]}\|(w_n, s_n)-(g_n^{\ast}(t), 0)\|_E\leq \frac{1}{\sqrt{n}},
\end{eqnarray}
by which we deduce that
\begin{eqnarray*}
\|(\widetilde{\phi},0)\|_E^2 &=& \|\widetilde{\phi}\|_X^2 = \int_{\R^N}|\phi(x)|^2dx + e^{-2s_n}\int_{\R^N}|\nabla \phi(x)|^2dx \\
&+&e^{-N s_n} \int_{\R^N}|\phi(x)|^4dx + e^{-(N+4)s_n}\int_{\R^N}|\nabla \phi(x)|^4dx \ \leq 4 \ \|\phi\|_X^2,
\end{eqnarray*}
holds as $n\in \N$ large enough. Thus \eqref{F'u} has been proved. Finally, we know from (ii) that for each $n\in \N$, there exists a $t_n\in [0,1]$, such that $\| (w_n, s_n)-(g_n^{\ast}(t_n), 0) \|_E \to 0$. This implies in particular that
$$\| w_n - g_n^{\ast}(t_n) \|_X \to 0.$$
Thus from \eqref{value-s_n} and
$$\| u_n - v_n \|_X = \| H(w_n, s_n) - g_n^{\ast}(t_n) \|_X \leq \| H(w_n, s_n) - w_n\|_X + \|w_n- g_n^{\ast}(t_n) \|_X,$$
we conclude that $\| u_n - v_n \|_X \to 0$ as $n\to \infty$. At this point, the proof of the lemma is completed.
\end{proof}

We now derive a similar Palais-Smale sequence for $J_{\mu}$ at the level $\tilde{m}_{\mu}(c)$. As a first step we prove

\begin{lem}\label{lm3}
For any given $c\in (c_0, c(p,N))$, there exists a minimizing sequence $\{u_n\}\subset \Sigma_c \setminus B$ of Schwarz symmetric functions, such that
\begin{eqnarray}\label{31}
\left\{
\begin{array}{l}
J_{\mu}(u_n)\to \tilde{m}_{\mu}(c),\\
Q_{\mu}(u_n)= 0,\ \forall\ n\in \N.\\
\end{array}
\right.
\end{eqnarray}
\end{lem}

\begin{proof}
First we prove that there exists a sequence $\{v_n\}\subset \Sigma_c \setminus B$ satisfying \eqref{31}. Let $\{u_n\}\subset \Sigma_c \setminus B$ be such that $J_{\mu}(u_n) \to \tilde{m}_{\mu}(c)$, we claim that we may assume that $\{u_n\}$ satisfies $Q_{\mu}(u_n)=0$ for all $n\in \N$.

Indeed, if $Q_{\mu}(u_n)=0$, for some $n\in \N$, we are done. If $Q_{\mu}(u_n)\neq 0$, then we consider the scaling
\begin{eqnarray}\label{scaling}
u_n^t(x):= t^{N/2}u_n(tx),\ \forall\ t>0.
\end{eqnarray}
Note that for all $t>0$, $u_n^t\in \Sigma_c$ and direct calculations show that $Q_{\mu}(u_n^t)= \frac{1}{t}\frac{d}{dt}J_{\mu}(u_n^t)$ and $Q_{\mu}(u_n^t) \to \infty$ as $t \to \infty$.   Thus if $Q_{\mu}(u_n)<0$ there exists by continuity a $t_n^0>1$ such that $Q_{\mu}(u_n^{t_n^0})=0$ and
\begin{eqnarray}\label{4}
J_{\mu}(u_n^{t_n^0})\leq J_{\mu}(u_n).
\end{eqnarray}
Also $u_n^{t_n^0}\in \Sigma_c \setminus B$. If $Q_{\mu}(u_n)>0$ we also claim that there exists a $t_n^0\in (0,1)$, such that $Q_{\mu}(u_n^{t_n^0})=0$ and $J_{\mu}(u_n^{t_n^0}) \leq J_\mu (u_n)$. To prove the claim first observe that it is not possible to have $Q_{\mu}(u_n^t)>0,\ \forall\ t\in (0,1)$ since otherwise there exists a $t_n^{\ast}\in (0,1)$ such that $\int_{\R^N}(1+|u_n^{t_n^{\ast}}|^2) |\nabla u_n^{t_n^{\ast}}|^2 dx=k_0$ and this  leads to
\begin{eqnarray}\label{5}
J_{\mu}(u_n^{t_n^{\ast}})\leq J_{\mu}(u_n)\to \tilde{m}_{\mu}(c)
\end{eqnarray}
and
\begin{eqnarray}\label{6}
J_{\mu}(u_n^{t_n^{\ast}})\geq \inf_{u\in C_{k_0}}J_{\mu}(u).
\end{eqnarray}
Clearly \eqref{5}-\eqref{6} contradict \eqref{minimum2}.  We conclude that there exists a $t_n^0\in (0,1)$ such that
$$Q_{\mu}(u_n^{t_n^0})=0\ \mbox{ and }\ J_{\mu}(u_n^{t_n^0})\leq J_{\mu}(u_n).$$
Since by Lemma \ref{estimates}, $Q_\mu(u) >0$ for $u \in B$ we also have that $u_n^{t_n^0}\in \Sigma_c\setminus B$. At this point we have shown that it is possible to choose, for each $n \in \N$ a $t_n^0 >0$ such that  $Q_{\mu}(u_n^{t_n^0})=0$, $J_{\mu}(u_n^{t_n^0})\to \tilde{m}_{\mu}(c)$ and $u_n^{t_n^0}\in \Sigma_c\setminus B$.
\vskip2mm

Now denote by $u_n^{\ast}$ the Schwarz symmetrization of $u_n$ and let us prove that $\{u_n^{\ast}\} \subset \Sigma_c \backslash B$ and is a minimizing sequence of $\tilde{m}_{\mu}(c)$. For each $n\in \N$, by the Polya-Szeg\"{o} inequality $\|\nabla u^{\ast}\|_q^q\leq \|\nabla u\|_q^q,\ \forall q\in [1,\infty)$ and also \cite[Lemma 4.3]{CJS} we have
\begin{eqnarray}\label{7}
J_{\mu}(u_n^{\ast})\leq J_{\mu}(u_n),\ \mbox{ and }\ Q_{\mu}(u_n^{\ast})\leq Q_{\mu}(u_n)=0.
\end{eqnarray}
At this point \eqref{lQ} implies that $u_n^{\ast}\in \Sigma_c\setminus B$. This and \eqref{7} lead to $$\left\{
\begin{array}{l}
J_{\mu}(u_n^{\ast})\to \tilde{m}_{\mu}(c),\\
Q_{\mu}(u_n^{\ast})\leq 0,\ \forall\ n\in \N.\\
\end{array}\right.
$$
If $Q_{\mu}(u_n^{\ast})<0$, we may use the above scaling arguments to get a $v_n^{\ast}\in \Sigma_c\setminus B$ satisfying \eqref{31}. At this point the proof is completed.
\end{proof}

\begin{lem}\label{lm4}[A special Palais-Smale sequence for $\tilde{m}_{\mu}(c)$]
Assume that $p\in (1+\frac{4}{N}, 3+\frac{4}{N}), N\geq 1$. Then for any given $c\in (c_0, c(p,N))$, for each $\mu\in (0,\mu_0)$ there exists a sequence $\{u_n\}\in \Sigma_c\setminus B$, and a sequence $\{v_n\}\subset \Sigma_c \backslash B$ of Schwarz symmetric functions, such that
\begin{equation}\label{8}
\left\{
\begin{array}{l}
J_{\mu}(u_n)\to \tilde{m}_{\mu}(c),\\
\|u_n-v_n\|_{X}\to 0,\\
\|J'_{\mu}|_{\Sigma_c}(u_n)\|_{X^{\ast}}\to 0,\\
Q_{\mu}(v_n)=0,\ \forall\ n\in \N,
\end{array}
\right.
\end{equation}
as $n\to \infty$.
\end{lem}
\begin{proof}
In Lemma \ref{lm3} we have obtained a sequence $\{v_n\}\subset \Sigma_c\setminus B$ of Schwarz symmetric functions, satisfying
$$J_{\mu}(v_n)\to \tilde{m}_{\mu}(c)\ \mbox{ and } Q_{\mu}(v_n)=0,\ \forall\ n\in \N. $$
It is standard to show that, for any $a >0$ there exists a $b >0$ such that
$$ J_\mu (u) \geq \inf_{u \in C_{k_0}}J_\mu(u) - a \quad \mbox{if} \quad u \in  \cup_{k \in [k_0-b, k_0 +b]} C_k.$$
Thus from \eqref{minimum2} we deduce that $\{v_n\} \subset \Sigma_c \backslash \bigcup_{0 < k \leq k_0 +b}$ for some $b>0$ and for
$n \in \N$ large enough. Thus, roughly speaking, $\{v_n\}$ stays away from the boundary. At this point we deduce, using Ekeland's variational principle,
that there exists a sequence $\{u_n\}\subset \Sigma_c$ such that
\begin{equation}\label{9}
\left\{
\begin{array}{l}
J_{\mu}(u_n)\leq J_{\mu}(v_n),\\
\|u_n-v_n\|_{X}\to 0,\\
\|J'_{\mu}|_{\Sigma_c}(u_n)\|_{X^{\ast}}\to 0.\\
\end{array}
\right.
\end{equation}
This completes the proof of the lemma.
\end{proof}

Next we show the compactness of the Palais-Smale sequences obtained in Lemmas  \ref{lm2} and \ref{lm4}.
\begin{prop}\label{c3-prop1}
Assume that $p\in (1+\frac{4}{N}, 3+\frac{4}{N}),$ $N \geq 1$. Let $\{u_n \}\subset \Sigma_c$ be a Palais-Smale sequence as obtained in Lemmas \ref{lm2} or \ref{lm4}. Then there exist a $u_{\mu}\in X \backslash \{0\} $ and a $\lambda_{\mu}\in \R$,  such that, up to a subsequence,
\begin{enumerate}
  \item [1)] $u_n\rightharpoonup  u_{\mu} > 0$,  in $X$;
  \item [2)] $J'_{\mu}(u_n)-\lambda_{\mu} u_n \to 0$,  in $X^{\ast}$;
  \item [3)] $J'_{\mu}(u_{\mu})-\lambda_{\mu} u_{\mu} = 0$,  in $X^{\ast}$.
\end{enumerate}
Moreover, if $\lambda_{\mu} < 0$, we have that
\begin{eqnarray}\label{conv1}
\lim_{n\to \infty}\|u_n-u_{\mu}\|_{X}=0.
\end{eqnarray}
\end{prop}

\begin{proof}[Proof of Proposition \ref{c3-prop1}]
From Lemma \ref{boundedness} we know that $\{u_n\}$ is bounded in $X$. This implies in particular the boundedness of the Schwarz symmetric sequences $\{v_n\}$ obtained in Lemmas \ref{lm2} or \ref{lm4}. Thus by \cite[Proposition 1.7.1]{TC} we conclude that up to a subsequence, there exists a $u_{\mu}\in X$, which is non-negative and Schwarz symmetric, such that
$$v_n \rightharpoonup u_{\mu}\geq 0,\ \mbox{ in } X;$$
\begin{eqnarray*}
v_n \to u_{\mu}, \ \mbox{ in } L^q(\R^N),\ \forall\ q\in (2, 2^{\ast}).
\end{eqnarray*}
By interpolation, we have that
\begin{eqnarray}\label{conv00}
v_n \to u_{\mu}, \ \mbox{ in } L^q(\R^N),\ \forall\ q\in (2, 2\cdot 2^{\ast}).
\end{eqnarray}
In view of $\|u_n-u_{\mu}\|_q \leq \|u_n-v_n\|_q+\|v_n-u_{\mu}\|_q$, one gets that
\begin{eqnarray}\label{conv2}
u_n \to u_{\mu}, \ \mbox{ in } L^q(\R^N),\ \forall\ q\in (2, 2\cdot 2^{\ast}).
\end{eqnarray}
At this point we shall use the additional information that $Q_{\mu}(u_n)\to 0$ or $Q_{\mu}(v_n) = 0,\ \forall\ n\in \N$ to show that $u_{\mu}\neq 0$. First let  $\{u_n\}\subset \Sigma_c$
be the Palais-Smale sequence constructed in Lemma \ref{lm2} and assume that  $u_{\mu} = 0$. Then by \eqref{conv2} we have that $||u_n||_{p+1} \to 0$ and using that $Q_{\mu}(u_n) \to 0$ we deduce that
$$\|\nabla u_n\|_4^4\ \to 0 \ \mbox{ and } \  \int_{\R^N}(1+|u_n|^2)|\nabla u_n|^2dx \to 0.$$
This leads to $J_{\mu}(u_n)\to 0$, which contradicts the fact that $J_{\mu}(u_n)\to \gamma_{\mu}(c)>0$. Now for the Palais-Smale sequence constructed in Lemma \ref{lm4} if we assume that $u_{\mu} =0$ we also get from \eqref{conv00} and $Q_{\mu}(v_n) = 0$ that
$$  \int_{\R^N}(1+|v_n|^2)|\nabla v_n|^2dx \to 0.$$
This contradicts the fact that $\{v_n\} \subset \Sigma_c \backslash B$. Having proved in both cases that $u_{\mu} \neq 0$, Point 1) is established. \medskip

Since $\{u_n\} \subset X$ is bounded, following Berestycki and Lions \cite[Lemma 3]{BELI}, we know that
\begin{eqnarray*}
J_{\mu}'|_{\Sigma_c}(u_n) &\longrightarrow& 0 \ \mbox{ in } X^{\ast} \\
&\Longleftrightarrow & J_{\mu}'(u_n)-\langle J_{\mu}'(u_n),u_n \rangle u_n \longrightarrow 0\ \mbox{ in } X^{\ast}.
\end{eqnarray*}
Thus for any $\phi\in X$,
\begin{eqnarray}\label{1200}
\langle J_{\mu}'(u_n) &-&\langle J_{\mu}'(u_n), u_n \rangle u_n, \phi  \rangle = \mu \int_{\R^N} |\nabla u_n|^2\nabla u_n \nabla \phi dx + \int_{\R^N} \nabla u_n \nabla \phi dx \nonumber \\
&+& 2 \int_{\R^N}\Big (  u_n \phi |\nabla u_n|^2 + | u_n|^2\nabla  u_n \nabla  \phi \Big ) dx -  \int_{\R^N} |u_n|^{p-1}u_n \phi dx  \\
&-& \lambda_n \int_{\R^N} u_n \phi dx  \longrightarrow 0,\nonumber
\end{eqnarray}
where
\begin{equation}\label{lambda2}
\lambda_n=\frac{1}{\|u_n\|_2^2}\Big\{ \mu \|\nabla u_n\|_4^4+\|\nabla u_n\|_2^2+4\int_{\R^N}|u_n|^2|\nabla u_n|^2dx-\int_{\R^N}|u_n|^{p+1}dx\Big\}.
\end{equation}

In particular $J'_{\mu}(u_n)u_n-\lambda_n \|u_n\|_2^2\to 0$, and it follows that $\{\lambda_n\}$ is bounded since  $$J'_{\mu}(u_n)u_n= \mu \|\nabla u_n\|_4^4+\|\nabla u_n\|_2^2+4\int_{\R^N}|u_n|^2|\nabla u_n|^2dx-\int_{\R^N}|u_n|^{p+1}dx$$
is bounded. Thus there exists a $\lambda_{\mu}\in \R$, such that up to a subsequence, $\lambda_n \to \lambda_{\mu}$.  This and \eqref{1200} imply Point 2). To check Point 3), it is enough, in view of Point 2), to show that for any $\phi\in X$,
\begin{eqnarray}\label{117}
\langle J'_{\mu}(u_n)- \lambda_{\mu} u_n, \phi \rangle \to \langle J'_{\mu}(u_{\mu})- \lambda_{\mu} u_{\mu}, \phi \rangle.
\end{eqnarray}
To prove \eqref{117}, note that
\begin{eqnarray*}
\langle J'_{\mu}(u_n)- \lambda_{\mu} u_n, \phi \rangle &=& \mu \int_{\R^N} |\nabla u_n|^2 \nabla u_n \nabla \phi dx  + \int_{\R^N} \nabla u_n \nabla \phi dx - \lambda_{\mu} \int_{\R^N} u_n \phi dx \\
&+& 2\int_{\R^N}  \Big (u_n \phi |\nabla u_n|^2 +  |u_n|^2\nabla u_n \nabla \phi \Big) dx - \int_{\R^N} |u_n|^{p-1}u_n \phi dx.
\end{eqnarray*}
Since $u_n \rightharpoonup u_{\mu}$ in $X$, clearly we have
$$\int_{\R^N} \nabla u_n \nabla \phi dx \to \int_{\R^N} \nabla u_{\mu} \nabla \phi dx,$$
$$\int_{\R^N} u_n \phi dx \to \int_{\R^N} u_{\mu} \phi dx,\  \int_{\R^N} |u_n|^{p-1}u_n \phi dx \to \int_{\R^N} |u_{\mu}|^{p-1}u_{\mu} \phi dx.$$
Thus we only need to prove that
\begin{equation}\label{118}
\int_{\R^N} |\nabla u_n|^2 \nabla u_n \nabla \phi dx \to \int_{\R^N} |\nabla u_{\mu}|^2 \nabla u_{\mu} \nabla \phi dx;
\end{equation}
\begin{equation}\label{119}
\int_{\R^N} u_n \phi |\nabla u_n|^2 + |u_n|^2\nabla u_n \nabla \phi dx \to \int_{\R^N} u_{\mu} \phi |\nabla u_{\mu}|^2 + |u_{\mu}|^2\nabla u_{\mu} \nabla \phi dx.
\end{equation}

But $\{|\nabla u_n|^2 \nabla u_n\}$ is bounded in $L^{4/3}(\R^N)$ since $\{\nabla u_n\}$ is bounded in $ L^4(\R^N)$. Thus $|\nabla u_n|^2 \nabla u_n \rightharpoonup |\nabla u_{\mu}|^2 \nabla u_{\mu}$ in $L^{4/3}(\R^N)$ and then we get \eqref{118}, by weak convergence for any $\nabla \phi \in  L^{4}(\R^N)$. Similarly,  using the Young inequality, we have that
$$ (|u_n| |\nabla u_n|^2)^{4/3} \leq \frac{1}{3} |u_n|^4 + \frac{2}{3}|\nabla u_n|^4,$$
$$(|u_n|^2 |\nabla u_n|)^{4/3} \leq \frac{2}{3} |u_n|^4 + \frac{1}{3}|\nabla u_n|^4.$$
These yield that both $\{|u_n| |\nabla u_n|^2\}$ and $\{|u_n|^2 |\nabla u_n|\}$ are bounded in $L^{4/3}(\R^N)$, since $\{u_n\}$ is bounded in $X$. Thus \eqref{119} holds by a similar argument. At this point, \eqref{117} holds and we have proved Point 3). Finally, we note from Points 2) and 3) that
$$\langle J'_{\mu}(u_n)-\lambda_{\mu} u_n, u_n\rangle\to \langle J'_{\mu}(u_{\mu})-\lambda_{\mu} u_{\mu}, u_{\mu} \rangle =0.$$
Using \eqref{conv2} we obtain that
\begin{eqnarray*}
\mu\|\nabla u_n\|_4^4&+&\|\nabla u_n\|_2^2-\lambda_\mu \|u_n\|_2^2+4 \int |u_n|^2|\nabla u_n|^2dx
\to \\
\mu\|\nabla u_\mu\|_4^4&+& \|\nabla u_\mu\|_2^2-\lambda_\mu \|u_\mu\|_2^2+4 \int |u_\mu|^2|\nabla u_\mu|^2dx.
\end{eqnarray*}
If $\lambda_{\mu}<0$, this implies that $\|u_n-u_{\mu}\|_X\to 0$, since $u_n\rightharpoonup u_{\mu}$ in $X$.
\end{proof}


Based on the above preliminary works, we conclude that

\begin{thm}\label{th3.1}
Assume that $p\in (1+\frac{4}{N}, 3+\frac{4}{N})$, $N \geq 1$. Then there exists a $c_0 \in (0, c(p,N))$ such that
\begin{enumerate}
\item [(1)]  For any $c \in (c_0, \infty)$ there exists a $\mu_0>0$ such that for each $\mu\in (0,\mu_0)$, the functional $J_{\mu}$ has a critical point $u_{\mu}$,  which is Schwarz symmetric and satisfies $J_\mu(u_{\mu})\leq \gamma_\mu (c)$, on $\Sigma_{\bar{c}}$ with $0<\bar{c} \leq c$, i.e., there exists a $\lambda_{\mu} \in \R $ such that $J'_{\mu}(u_{\mu})-\lambda_{\mu} u_{\mu}=0$. In addition, if $\lambda_{\mu}<0$, then we have $u_{\mu}\in \Sigma_c$ and $J_\mu(u_{\mu}) = \gamma_\mu (c)$.
\vspace{2mm}
\item [(2)] For any $c \in (c_0, c(p,N))$ there exists a $\mu_0>0$ such that for each $\mu\in (0,\mu_0)$, the functional $J_{\mu}$ has a critical point $v_{\mu}$ on $\Sigma_c$
which is Schwarz symmetric and satisfies $J_\mu(v_{\mu})= \tilde{m}_\mu (c)$, i.e., there exists a $\beta_\mu <0$ such that $J'_{\mu}(v_{\mu})-\beta_\mu v_{\mu}=0$. In addition, $\beta_{\mu} \leq \bar \beta <0$ for some $ \bar \beta <0$ independent of $\mu$.
\vspace{2mm}
\item [(3)] In the both cases, $Q_{\mu}(u_{\mu})=0$ and $Q_{\mu}(v_{\mu})=0$ hold, where the functional $Q_{\mu}$ is given by \eqref{defQ}.
\end{enumerate}
\end{thm}

\begin{proof}
Clearly, Point $(1)$ follows directly from Lemma \ref{lm2} and Proposition \ref{c3-prop1}. To show Point $(2)$, we note first that, from Lemma \ref{lm4} and Proposition \ref{c3-prop1} one can deduce that for any $c \in (c_0, c(p,N))$ there exists a $\mu_0>0$ such that for each $\mu\in (0,\mu_0)$, the functional $J_{\mu}$ has a critical point $v_{\mu}$, which is Schwarz symmetric and satisfies $J_\mu(v_{\mu})\leq \tilde{m}_\mu (c)$, on $\Sigma_{\bar{c}}$ with $0<\bar{c}\leq c$, i.e., there exists a $\beta_{\mu} \in \R $ such that $J'_{\mu}(v_{\mu})-\beta_{\mu} v_{\mu}=0$.    Moreover, if $\beta_{\mu}<0$, then we have $v_{\mu}\in \Sigma_c$ and $J_\mu(v_{\mu})= \tilde{m}_\mu (c)$. Thus we only need to verify that $\beta_{\mu} \leq \bar \beta <0$ for some $ \bar \beta <0$ independent of $\mu >0$.

On one hand, following the proofs of \cite[Lemma 5.10]{LWW2} and \cite[Lemma 3.1]{CJS}, we see that $(v_{\mu}, \beta_{\mu})$ satisfies the identity
\begin{equation}\label{pppp1}
\frac{\mu(N-4)}{4N}\|\nabla v_{\mu}\|_4^4 + \frac{N-2}{N} \left(\frac{1}{2}\|\nabla v_{\mu}\|_2^2+\int_{\R^N}|v_{\mu}|^2|\nabla v_{\mu}|^2dx  \right) -\frac{\beta_{\mu}}{2}\|v_{\mu}\|_2^2 = \frac{1}{p+1}\|v_{\mu}\|_{p+1}^{p+1}.
\end{equation}
From this identity, we obtain that
\begin{equation}\label{ones}
\tilde{m}_\mu (c) \geq J_{\mu}(v_{\mu})=  \frac{\mu}{N} \|\nabla v_{\mu}\|_4^4  + \frac{1}{N}\Big( \|\nabla v_{\mu}\|_2^2 + 2\int_{\R^N}|v_{\mu}|^2 |\nabla v_{\mu}|^2 dx\Big) +  \frac{\beta_{\mu}}{2} \|v_{\mu}\|_2^2.
\end{equation}
Now from the proof of Lemma \ref{lem1} and the definition of $\tilde{m}_\mu (c)$, we know that $\tilde{m}_\mu (c) \to 0 $ if $c_0 \to c(p,N)$ and $\mu\to 0$. Thus since $J_\mu(v_{\mu})\leq  \tilde{m}_\mu (c)$ we can assume that $J_\mu(v_{\mu})$ is arbitrarily close to $0$ (or negative). On the other hand, we claim that there exists a constant $C >0$, independent of $v_{\mu}$ and $c>0$, such that
\begin{eqnarray}\label{estimate}
\int_{\R^N} (1+2|v_{\mu}|^2)|\nabla v_{\mu}|^2dx \geq C>0.
\end{eqnarray}
Indeed, testing the equation $J'_{\mu}(v_{\mu})-\beta_{\mu} v_{\mu}=0$ by $v_{\mu}$, we obtain
\begin{equation}\label{pppp2}
\mu\|\nabla v_{\mu}\|_4^4 + \|\nabla v_{\mu}\|_2^2 + 4\int_{\R^N}|v_{\mu}|^2|\nabla v_{\mu}|^2dx - \beta_{\mu}\|v_{\mu}\|_2^2 - \|v_{\mu}\|_{p+1}^{p+1}=0.
\end{equation}
Thus from \eqref{pppp1} and \eqref{pppp2} we can derive that $Q_{\mu}(v_{\mu})=0$. Then \eqref{estimate} follows immediately from  \eqref{lQ}. Combining \eqref{ones} and \eqref{estimate} we see that, by letting $c_0$ be sufficiently close to $c(p,N)$ and $\mu>0$ small enough, we can insure that
$\beta_{\mu}< \bar \beta <0$ for some $\bar \beta <0$  independent of $\mu$. Then the proof of Point (2) is completed. Finally as for $Q_{\mu}(v_{\mu})=0$ can also show that $Q_{\mu}(u_{\mu})=0$. Thus Point (3) holds.
\end{proof}


\section{Convergence issues}

In this section, letting $\mu \to 0$, we show that the sequences of critical points of $J_{\mu}$ obtained in Theorem \ref{th3.1} converge to critical points of $J=J_{0}$ on $S(c)$.

\begin{thm}\label{th2}[Convergence issues] Assume that $c>0 $ is fixed and that $\mu_n \to 0$ as $n \to \infty$. Let $\{w_n\}\subset \Sigma_{c_n}$ be a sequence of Schwarz symmetric functions, and $\{\lambda_n\} \subset \R$,  satisfying that
$$0<\delta_0\leq c_n \leq c,\ |J_{\mu_n}(w_n)|\leq C,\ \mbox{ and }\ J'_{\mu_n}(w_n)-\lambda_n w_n=0,$$
where $\delta_0>0,$ $ C>0$ are independent of $n\in \N$. Then there exist a $w_c\in W^{1,2}\cap L^{\infty}(\R^N)\setminus \{0\}$ and a $\lambda_c\in \R$, such that  up to a subsequence, as $n\to \infty$, we have that
\begin{eqnarray}\label{101}
\lambda_n \to \lambda_c, \mbox{ in } \R,\nonumber\\
\ J'(w_c)-\lambda_c w_c = 0.
\end{eqnarray}
Moreover, if $\lambda_c<0$, then
\begin{eqnarray}\label{102}
w_n &\to& w_c,\  \mbox{ in } W^{1,2}(\R^N),\nonumber\\
w_n\nabla w_n &\to& w_c \nabla w_c,\  \mbox{ in } L^2(\R^N),\\
\mu_n \|\nabla w_n\|_4^4 &\to& 0,\nonumber
\end{eqnarray}
as $n\to \infty$. Thus $w_c\in W^{1,2}\cap L^{\infty}(\R^N)$ is a critical point of $J$ on $ S(c')$ where $c'=\lim_{n\to \infty}c_n $.

\end{thm}
\begin{proof}
To show this theorem, we borrow ideas from the proof of \cite[Theorem 1.1]{LLW}. First, since $0<\delta_0\leq c_n \leq c$, $|J_{\mu_n}(w_n)|\leq C$ and $\ J'_{\mu_n}(w_n)-\lambda_n w_n=0$, we observe from the proofs of Lemma \ref{boundedness} and Proposition \ref{c3-prop1} that $\{\int_{\R^N}(1+|w_n|^2)|\nabla w_n|^2dx\}$ is bounded and $\{\lambda_n\}$ is bounded. Thus up to a subsequence, $\lambda_n \to \lambda_c \in \R$, and noting that $\{w_n\}\subset \Sigma_c$ is Schwarz symmetric, by \cite[Proposition 1.7.1]{TC} we obtain, up to a subsequence that
\begin{eqnarray}\label{110}
w_n &\rightharpoonup& w_c,\ \mbox{ in } W^{1,2}(\R^N),\nonumber\\
w_n &\to& w_c,\ \mbox{ in } L^q(\R^N),\ \forall q\in (2, 2\cdot 2^{\ast}),\\
w_n \nabla w_n &\rightharpoonup& w_c \nabla w_c,\ \mbox{ in } L^2(\R^N),\nonumber \\
w_n &\to& w_c,\ \mbox{a.e. in } \R^N,\nonumber
\end{eqnarray}
for some $w_c\in W^{1,4}\cap W^{1,2}(\R^N)$.  Since $\{w_n\}$ satisfies $J'_{\mu_n}(w_n)-\lambda_n w_n=0$,  we have
\begin{eqnarray}\label{eq1.4}
\mu_n\int_{\R^N} |\nabla w_n|^2\nabla w_n \nabla \phi dx + \int_{\R^N} \nabla w_n \nabla \phi dx - \lambda_n \int_{\R^N} w_n \phi dx  \nonumber \\
+ 2 \int_{\R^N} \Big ( w_n \phi |\nabla w_n|^2 + | w_n|^2\nabla  w_n \nabla \phi \Big ) dx  = \int_{\R^N} |w_n|^{p-1}w_n \phi dx,
\end{eqnarray}
for any $\phi \in W^{1,4}\cap W^{1,2}(\R^N).$ Then by the Sobolev inequality and Moser iteration we may get
\begin{eqnarray}\label{1100}
\|w_n\|_{L^{\infty}(\R^N)}\leq C,\ \mbox{ and } \|w_c\|_{L^{\infty}(\R^N)}\leq C,
\end{eqnarray}
for some constant $C>0$.\\

We now show that $w_c$ satisfies that
\begin{eqnarray}\label{111}
\langle J'(w_c)-\lambda_c w_c, \phi  \rangle =0,\ \forall\ \phi \in W^{1,2}\cap L^{\infty}(\R^N).
\end{eqnarray}

In \eqref{eq1.4} we choose $\phi = \psi \exp(-w_n)$ with $\psi\in C_0^{\infty}(\R^N)$, $\psi\geq 0$. Then we have that
\begin{eqnarray*}
0&=& \mu_n\int_{\R^N} |\nabla w_n|^2 \nabla w_n (\nabla \psi \exp(-w_n) - \psi \exp (-w_n)\nabla w_n )dx\\
&+&\int_{\R^N} \nabla w_n (\nabla \psi \exp(-w_n) - \psi \exp (-w_n)\nabla w_n )dx \\
&+& 2 \int_{\R^N} |w_n|^2 \nabla w_n (\nabla \psi \exp(-w_n) - \psi \exp (-w_n)\nabla w_n )dx \\
&+& 2 \int_{\R^N} w_n \psi \exp(-w_n)|\nabla w_n|^2dx - \lambda_n \int_{\R^N} u_n \psi \exp(- w_n)dx \\
&-& \int_{\R^N} |w_n|^{p-1}w_n \psi \exp (-w_n)dx.
\end{eqnarray*}
This implies that
\begin{eqnarray*}
0&\leq& \mu_n\int_{\R^N} |\nabla w_n|^2 \nabla w_n \nabla \psi \exp(-w_n)dx +\int_{\R^N} (1+2 w_n^2)\nabla \psi \nabla w_n \exp(-w_n)dx\\
 &-& \int_{\R^N}(1+2 w_n^2-2 w_n)\psi \exp (-w_n)|\nabla w_n|^2 dx \\
&-& \lambda_n \int_{\R^N} w_n \psi \exp(- w_n)dx - \int_{\R^N} |w_n|^{p-1}w_n \psi \exp (-w_n)dx.
\end{eqnarray*}
By using \eqref{110} and the fact that $\{\mu_n\|\nabla w_n\|_4^4\}$ is bounded, we deduce that
\begin{eqnarray}\label{112}
\int_{\R^N} \nabla w_c \nabla \left(\psi \exp(- w_c) \right ) dx + \int_{\R^N} 2 |w_c|^2 \nabla w_c \nabla \left(\psi \exp(- w_c)\right )dx \nonumber \\
+ \int_{\R^N} 2 w_c \left( \psi \exp(- w_c) \right ) |\nabla w_c|^2 dx
- \lambda_c\int_{\R^N} w_c \left( \psi \exp(- w_c)\right ) dx \nonumber \\
\geq \int_{\R^N} |w_c|^{p-1}w_c \left(\psi \exp(- w_c)\right ) dx,
\end{eqnarray}
in which we also used Fatou's lemma, to get
$$\liminf_{n\to \infty} \int_{\R^N} (1+ 2w_n^2-2 w_n)|\nabla w_n|\psi \exp(-w_n)dx \geq \int_{\R^N} (1+ 2w_c^2-2 w_c)|\nabla w_c|\psi \exp(-w_c)dx.$$

Let $\chi \geq 0, \chi \in C_0^{\infty}(\R^N)$. Choose a sequence of non-negative functions $\psi_n\in C_0^{\infty}(\R^N)$ such that $\psi_n \to \chi \exp(w_c)$ in $W^{1,2}(\R^N)$, $\psi_n \to \chi \exp(w_c)$ a.e. in $\R^N$, and $\psi_n$ is uniformly bounded in $L^{\infty}(\R^N)$. Then we get from \eqref{112} that
$$\int_{\R^N} \nabla w_c \nabla \chi dx + 2\int_{\R^N}\Big( |w_c|^2\nabla w_c \nabla \chi+ w_c \chi |\nabla w_c|^2 \Big)dx- \lambda_c \int_{\R^N} w_c \chi dx\geq \int_{\R^N} |w_c|^{p-1}w_c \chi dx.$$
Similarly by choosing $\phi= \psi \exp(w_n)$, we get an opposite inequality. Thus we obtain that for any $\chi \in C_0^{\infty}(\R^N)$,
\begin{eqnarray}\label{113}
\int_{\R^N} \nabla w_c \nabla \chi dx + 2\int_{\R^N} \Big(|w_c|^2 \nabla w_c \nabla \chi + w_c \chi |\nabla w_c|^2 \Big) dx\nonumber \\
- \lambda_c\int_{\R^N} w_c \chi dx = \int_{\R^N} |w_c|^{p-1}w_c \chi dx.
\end{eqnarray}
This proves \eqref{101}. \\

Now by approximation again, we get from \eqref{113} that
\begin{eqnarray}\label{114}
\qquad \int_{\R^N} |\nabla w_c|^2dx + 4\int_{\R^N} |w_c|^2 |\nabla w_c|^2dx - \lambda_c\int_{\R^N} |w_c|^2dx= \int_{\R^N} |w_c|^{p+1}dx.
\end{eqnarray}
In \eqref{eq1.4}, we use $\phi= w_n$ to get that
\begin{eqnarray}\label{115}
\mu_n\int_{\R^N} |\nabla w_n|^4dx + \int_{\R^N} |\nabla w_n|^2 dx + 4 \int_{\R^N} |w_n|^2 |\nabla w_n|^2 dx\nonumber  \\ - \lambda_n \int_{\R^N} |w_n|^2 dx = \int_{\R^N} |w_n|^{p+1}dx.
\end{eqnarray}
and then
\begin{eqnarray}\label{116}
\mu_n\int_{\R^N} |\nabla w_n|^4dx + \int_{\R^N} |\nabla w_n|^2 dx + 4 \int_{\R^N} |w_n|^2 |\nabla w_n|^2 dx\nonumber \\
- \lambda_c \int_{\R^N} |w_n|^2 dx = \int_{\R^N} |w_n|^{p+1}dx+o(1),
\end{eqnarray}
since $\lambda_n\to \lambda_c$ and $\lim_{n\to \infty} \int_{\R^N} |w_n|^2 dx=c'\geq \delta_0>0$. Hence, if $\lambda_c <0$, using $\int_{\R^N} |w_n|^{p+1}dx \to \int_{\R^N} |w_c|^{p+1}dx$ in \eqref{110}, we conclude from \eqref{110}, \eqref{114} and \eqref{116} that, as
$n\to \infty $,
$$\mu_n \int_{\R^N} |\nabla w_n|^4 dx \to 0,\ \int_{\R^N} |\nabla w_n|^2 dx \to \int_{\R^N} |\nabla w_c|^2 dx, $$
$$\int_{\R^N} |w_n|^2 |\nabla w_n|^2 dx\to \int_{\R^N} |w_c|^2 |\nabla w_c|^2 dx,\ \int_{\R^N} |w_n|^2 dx \to \int_{\R^N} |w_c|^2 dx.$$
Thus from \eqref{101} and \eqref{1100} we deduce that $w_c\in W^{1,2}\cap L^{\infty}(\R^N)\setminus \{0\}$ is a critical point of $J$ on $S(c')$.  At this point, the proof is completed.
\end{proof}

Now we are able to end the proof of Theorem \ref{mainresult1}.

\begin{proof}[Proof of Theorem \ref{mainresult1}]
In the case $c \in [c(p,N), \infty)$ the critical point $v_c$ is just the global minimizer already obtained in \cite{CJS, JL2} whose existence is recalled in Lemma \ref{lm0}. For the other cases let us first prove that there exists a $C>0$ independent of $\mu>0$, such that
\begin{equation}\label{1111}
|J_{\mu}(u_\mu)|\leq C \quad \mbox{and} \quad |J_{\mu}(v_\mu)|\leq C
\end{equation}
where $u_\mu$ and $v_\mu$ are obtained in Theorem \ref{th3.1}. To prove \eqref{1111}, note that by definition of $\gamma_{\mu}(c)$ we have $0<J_{\mu}(u_\mu) = \gamma_{\mu}(c) \leq \gamma_1(c)$, in which $\gamma_1(c)$ is independent of $\mu>0 $. Also when $c \in (c_0, c(p,N))$ we have $0 \leq J_\mu (v_\mu) \leq J_\mu(u_\mu) \leq \gamma_1(c)$. Observe that in the case of $v_{\mu}$ the conclusion of Theorem \ref{mainresult1} follows directly from Theorems \ref{th3.1} (2) and \ref{th2}. In the case of $u_{\mu}$ fix  $c>0$ and take $\mu_n\to 0$. By Theorem \ref{th3.1} there exists a sequence of Schwarz symmetric functions $w_n$ on $ S(c_n)$ and $\lambda_n\in \R$ such that $0<c_n\leq c$, $J_{\mu_n}(w_n)\leq \gamma_{\mu_n}(c)$, $J_{\mu_n}'(w_n)-\lambda_n w_n=0$ and $Q_{\mu_n}(w_n)= 0$.
We claim that $c_n\geq \delta_0$ for some $\delta_0>0$. In fact if $c_n\to 0$ we see from \eqref{pf4.5.1} that $w_n\to 0$ in $L^{p+1}(\R^N)$.
Using the fact that $Q_{\mu_n}(w_n)= 0$ we then deduce that $\int_{\R^N}(1+w_n^2) |\nabla w_n|^2 dx \to 0$. This is a contradiction since \eqref{lQ}. Then the claim is proved. Now we can apply Theorem \ref{th2} to $\{w_n\}$ so there exist $\lambda_c\in \R$ and $w_c\neq 0$ such that $w_n\to w_c$ in $L^{p+1}(\R^N)$, $\liminf_{n\to \infty} ||w_n||^2_2 \geq ||w_c||^2_2$, and $
J'(w_c) - \lambda_c w_c=0$.  Thus, by Lemmas \ref{lem2.1}, \ref{lem2.22} and \ref{lem2} the equation $J'(w_c) - \lambda_c w_c=0$ can not have Schwarz symmetric solutions in $L^2(\R^N)$ for $\lambda_c\geq 0$. Thus $\lambda_c<0$. Going back we may say that $\lambda_n<0$ for $n$ large (or $\mu_n$ small). Then by Theorem \ref{th3.1} (1) $c_n=c$ and $w_n\in S(c)$ and $J_{\mu_n}(w_n)= \gamma_{\mu_n}(c)$ for all $n$ large. Using Theorem \ref{th2} again we have $w_c\in S(c) $ and that $w_c\in W^{1,2}\cap L^{\infty}(\R^N)$ is a critical point of $J$ on $S(c)$. The proof of Theorem \ref{mainresult1} in now completed.
\end{proof}

\begin{proof}[Proof of Lemma \ref{asymLagrange}]
Fix a $c_0>0$ large and let $v_0 \in S(c_0)$ be fixed. We consider for
$t>0$ the scaling $v_0^{t}(x):= t^{\alpha}v_0(t^{\beta}x)$, where $$\alpha= \frac{1}{3N+4-Np}, \quad \beta=\frac{p-(3+\frac{2}{N})}{3N+4-Np}.$$
Then
$$\|v_0^t\|_2^2= t\|v_0\|_2^2,\quad \|\nabla v_0^t\|_2^2= t^{\lambda_1+1}\|\nabla v_0\|_2^2 ,$$
$$\int_{\R^N}|v_0^t|^2|\nabla v_0^t|^2dx= t^{\lambda_2+1}\int_{\R^N}|v_0|^2|\nabla v_0|^2dx,$$ $$\int_{\R^N}|v_0^t|^{p+1}dx=t^{\lambda_3+1}\int_{\R^N}|v_0|^{p+1}dx,$$
where
$$\lambda_1=\frac{2p-6-\frac{4}{N}}{3N+4-Np},\quad  \lambda_2=\frac{2p-4-\frac{4}{N}}{3N+4-Np},\quad \mbox{and} \quad \lambda_3=\frac{p-1}{3N+4-Np}.$$

We observe that $\lambda_3>0$ and $\lambda_3> \max\{\lambda_1, \lambda_2\}$ if $p\in (1, 3+\frac{4}{N})$. Also $v_0^t\in S(tc_0)$, for all $t>0$, and
\begin{equation*}
\frac{J(v_0^t)}{tc_0}= \frac{1}{c_0}\cdot \left( \frac{t^{\lambda_1}}{2}\|\nabla v_0\|_2^2+ t^{\lambda_2}\int_{\R^N}|v_0|^2|\nabla v_0|^2dx- \frac{t^{\lambda_3}}{p+1}\int_{\R^N}|v_0|^{p+1}dx \right).
\end{equation*}
This implies that
$$\limsup_{t\to \infty}\frac{m(tc_0)}{tc_0}\leq \limsup_{t\to \infty}\frac{J(v_0^t)}{tc_0} = -\infty,$$
from which we deduce that if $v_c \in S(c)$ is a global minimizer of $J$ on $S(c)$ then $J(v_c) c^{-1} \to - \infty$ as $c \to \infty.$
At this point recalling, see the proof of
\cite[Lemma 4.6]{CJS}, that
$$J(v_c)= \frac{1}{N}\left(\|\nabla v_c\|_2^2 + \int_{\R^N}|v_c|^2|\nabla v_c|^2dx \right)+\frac{\beta_c}{2}\|v_c\|_2^2$$
we deduce that $\beta_c \to - \infty$ as $c \to \infty$ uniformly.
\end{proof}

\begin{remark}
Concerning the behavior of the Lagrange multiplier $\lambda_c$ associated with the mountain pass solution we conjecture that $\lambda_c \to 0$ as $c \to \infty$.
\end{remark}


\section{Relationship between ground states and global minimizers on the constraint}

In this section, we prove Theorem \ref{th1.2} which gives a relationship between the ground states of \eqref{eq1.2} and the global minimizers of $m(c)$.
\vskip2mm

We recall from Lemma \ref{lm0} and \cite[Lemma 4.6]{CJS} that when $(p,c,N)$ satisfies the following conditions:
\begin{enumerate}
  \item [(i)] $c\in (0,\infty)$, and $p\in (1, 1+\frac{4}{N}), N\geq 1$,
  \vspace{1mm}
  \item [(ii)] $c\in [c(p,N), \infty)$, and $p\in (1+\frac{4}{N}, 3+\frac{4}{N}), N\geq 1$,
\end{enumerate}
there exist a  global minimizer $v_c$ of $m(c)$ and a Lagrange multiplier $\beta_c<0$, such that $(v_c, \beta_c)$ is a solution of \eqref{eq1.2}. Also we know from \cite[Theorem 1.3]{CJS} that for $\lambda=\beta_c<0$, the equation \eqref{eq1.2} has a ground state solution. We denote
$$\mathcal{A}_{\lambda}:=\Big\{u:\ u \mbox{ is a solution of \eqref{eq1.2}} \Big\},$$
$$\mathcal{G}_{\lambda}:=\Big\{u:\ u \mbox{ is a ground state solution of \eqref{eq1.2}} \Big\}.$$

\begin{proof}[Proof of Theorem \ref{th1.2}]
For $\lambda=\beta_c<0$, let $\varphi_{\beta_c}$ be a ground state of \eqref{eq1.2}. Namely, $\varphi_{\beta_c}$ solves the minimization problem
$$l_{\beta_c}:=\inf \{ I_{\beta_c}(u):\ u\in \mathcal{A}_{\beta_c} \},$$
Since $v_c \in \mathcal{A}_{\beta_c}$, one only needs to show that
\begin{eqnarray}\label{411}
I_{\beta_c}(v_c)=l_{\beta_c}.
\end{eqnarray}
By definition of $l_{\beta_c}$, to check \eqref{411} it is enough to show that $I_{\beta_c}(v_c) \leq l_{\beta_c}$. In turn this holds if one can find a $\psi \in S(c)$ such that
\begin{eqnarray}\label{412}
I_{\beta_c}(\psi)\leq I_{\beta_c}(\varphi_{\beta_c}),
\end{eqnarray}
since $I_{\beta_c}(v_c)\leq I_{\beta_c}(\psi)\leq I_{\beta_c}(\varphi_{\beta_c})=l_{\beta_c}$.
\vskip2mm

To choose $\psi \in S(c)$ satisfying \eqref{412}, we consider the scaling $u_t(x):=\varphi_{\beta_c}(x/t), t>0$. Then $\|u_t\|_2^2=t^N\|\varphi_{\beta_c}\|_2^2$ and by the identities \eqref{poho2.1} and \eqref{poho2.2}, we have
$$I_{\beta_c}(u_t)= \Big( t^{N-2}-\frac{N-2}{N}t^N \Big)\cdot \Big[ \frac{1}{2}\|\nabla \varphi_{\beta_c}\|_2^2+ \int_{\R^N}|\varphi_{\beta_c}|^2|\nabla \varphi_{\beta_c}|^2dx \Big].$$
Thus
$$\frac{d}{dt}I_{\beta_c}(u_t)=(N-2)(1-t^2)t^{N-3}\left [\frac{1}{2}\|\nabla \varphi_{\beta_c}\|_2^2+ \int_{\R^N}|\varphi_{\beta_c}|^2|\nabla \varphi_{\beta_c}|^2dx \right ].$$
This implies that when $N \geq 2$,
$$I_{\beta_c}(u_t)\leq I_{\beta_c}(\varphi_{\beta_c}),\ \forall t>0.$$
Choosing a suitable $t_0>0$ such that $u_{t_0}\in S(c)$ and letting $\psi=u_{t_0}$ we obtain \eqref{412}. When $N=1$, since by \cite[Theorem 1.3]{CJS} the non-negative solutions of \eqref{eq1.2} for fixed $\lambda >0$ are unique the conclusion holds automatically. This completes the proof.
\end{proof}

\begin{remark}\label{RNC}
The converse of Theorem \ref{th1.2} does not hold. Indeed on one hand our mountain pass solution is non-negative. On the other hand we know in several cases that there exists a unique non-negative solution to \eqref{eq1.2} when $\lambda <0$ is fixed. This is the case in particular when $N=1$, see \cite[Theorem 1.3]{CJS} (see also \cite{AZ}). Thus when this uniqueness property holds our mountain pass solution must be a ground state. This observation shows us in particular that not all ground states of \eqref{eq1.2} for $p\in (1+\frac{4}{N}, 3+\frac{4}{N})$ can be obtained as minimizers of $J$ on the corresponding constraint.
\end{remark}

\begin{remark}\label{symmetry}
From Theorem \ref{th1.2} we can deduce in particular that any global minimizer $v_c$ has a given sign and that $|v_c|$ is a radially symmetric, decreasing function with respect to one point. This follows directly from  \cite[Theorem 1.3]{CJS}.
\end{remark}


\section{Appendix}

In this section we give the proof of Lemma \ref{lm-decay} which was given to us by T. Watanabe \cite{Wa}.

\begin{proof}[Proof of Lemma \ref{lm-decay}] To show the lemma, we use arguments from \cite[Theorem 5]{MFSM}. \medskip

First we observe that from \eqref{semi-eq-0}, $v_0$ satisfies the  ODE
\begin{equation}\label{ODE}
v_0'' + \frac{N-1}{r}v_0' + f(v_0)^pf'(v_0) = 0.
\end{equation}
Multiplying \eqref{ODE} by $2r^Nv_0'$ we get
$$r^N\Big(2v_0'' v_0' + 2 f(v_0)^pf'(v_0)v_0' \Big) + 2 (N-1)r^{N-1}|v_0'|^2=0,$$
and hence
$$r^N\Big( |v_0'|^2 + \frac{2}{p+1}f(v_0)^{p+1}  \Big)' + 2(N-1)r^{N-1}|v_0'|^2=0.$$
Thus we obtain the following identity:
\begin{eqnarray}\label{5.1.1}
 \\
\Big(r^N\Big(|v_0'|^2+ \frac{2}{p+1}f(v_0)^{p+1}\Big) \Big)' + r^{N-1}\Big( (N-2)|v_0'|^2 - \frac{2N}{p+1}f(v_0)^{p+1} \Big)=0.\nonumber
\end{eqnarray}
We put
\begin{eqnarray*}
a(r)&:=&  r^N\Big(|v_0'|^2+ \frac{2}{p+1}f(v_0(r))^{p+1}\Big),\\
b(r)&:=&  \frac{2}{p+1}r^Nf(v_0(r))^{p+1}.
\end{eqnarray*}
Then from (\ref{5.1.1}) we obtain
\begin{eqnarray}\label{5.1.2}
a'(r)= r^{N-1}\Big(\frac{2N}{p+1}f(v_0)^{p+1} - (N-2)|v_0'|^2 \Big),
\end{eqnarray}
\begin{eqnarray}\label{5.1.3}
\Big(r^{N-2}a(r)\Big)'= 2(N-1)r^{N-3}b(r).
\end{eqnarray}
Next from \cite[Lemma 2.1]{CJ}, we know that there exists $C>0$ such that
$$
|f(s)|\leq  C |s|\ \mbox{ for } |s|\leq 1 \quad \mbox{ and } \quad
|f(s)|\leq  C |s|^{1/2} \mbox{ for } |s|\geq 1.$$
Since $\frac{2N}{N-2}<p+1<\frac{4N}{N-2}$ it follows that
\begin{eqnarray*}
|f(s)|^{p+1} \leq C |s|^{p+1} \leq C |s|^{\frac{2N}{N-2}} \ \mbox{ for } |s|\leq 1,\\
|f(s)|^{p+1} \leq C |s|^{(p+1)/2} \leq C |s|^{\frac{2N}{N-2}} \ \mbox{ for } |s|\geq 1,
\end{eqnarray*}
and hence
\begin{eqnarray}\label{5.1.4}
|f(s)|^{p+1} \leq C |s|^{\frac{2N}{N-2}} \ \mbox{ for all } s\geq 0.
\end{eqnarray}
Next since $v_0\in \mathcal{D}^{1,2}(\R^N)$ is a radially symmetric function, it follows that
\begin{eqnarray}\label{5.1.5}
|S^{N-1}|\int_0^{\infty}r^{N-1}|v_0'|^2dr = \int_{\R^N}|\nabla v_0|^2dx < \infty.
\end{eqnarray}
Then by the Sobolev inequality and from \eqref{5.1.4}, we obtain
\begin{eqnarray}\label{5.1.6}
|S^{N-1}|\int_0^{\infty}r^{N-1}f(v_0)^{p+1}dr & \leq & C |S^{N-1}|\int_0^{\infty}r^{N-1}|v_0|^{\frac{2N}{N-2}}dr\nonumber \\
&=& C \int_{\R^N}|v_0|^{\frac{2N}{N-2}}dx \nonumber \\
&\leq& C\Big( \int_{\R^N}|\nabla v_0|^2dx\Big)^{\frac{N}{N-2}}< \infty.
\end{eqnarray}

Now from \eqref{5.1.2}, \eqref{5.1.5} and \eqref{5.1.6}, it follows that $a'(r)\in L^{1}(0, \infty)$. Thus there exists $a_{\infty}\in [0, \infty)$ such that
$$\lim_{r\to \infty}a(r)= a_{\infty}.$$
Moreover since
$$\int_0^{\infty}\frac{a(r)}{r}= \int_0^{\infty}r^{N-1}\Big( |v_0'|^2 +\frac{2}{p+1}f(v_0)^{p+1}  \Big)dr < \infty,$$
it follows that $a_{\infty}=0$. Thus we have
\begin{eqnarray}\label{5.1.7}
\lim_{r\to \infty}a(r)= 0.
\end{eqnarray}
We define the non-increasing function $A(r)$ by
$$A(r):= \sup_{R\geq r}a(R).$$
Then from \eqref{5.1.7}, it follows that $A(r)\to 0$ as $r\to \infty$.

Next by the definition of $a(r)$ and since $f(v_0(r))^{p+1}>0$, we have
$$a(r)> r^{N}|v_0'(r)|^2,\quad |v_0'(r)|<\Big(\frac{a(r)}{r^N}\Big)^{1/2}.$$
Since $\lim_{r\to \infty}v_0(r)=0$  and $A(r)$ is non-increasing, we obtain
\begin{eqnarray*}
v_0(r) &=& - \int_{r}^{\infty} v_0'(R)dR\\
&\leq & \int_{r}^{\infty}\Big(\frac{a(R)}{R^N}\Big)^{1/2}dR \leq C A(r)^{1/2}r^{-\frac{N-2}{2}}.
\end{eqnarray*}
Thus by the definition of $b(r)$, it follows that
\begin{eqnarray}\label{5.1.8}
b(r)\leq C r^{N}v_0(r)^{\frac{2N}{N-2}}\leq CA(r)^{\frac{N}{N-2}}.
\end{eqnarray}
Next we integrate \eqref{5.1.3} over $(r, R)$. Since $A(r)$ is non-increasing, we have
\begin{eqnarray*}
R^{N-2}a(R)- r^{N-2}a(r) &=& 2(N-1)\int_{r}^R s^{N-3}b(s)ds\\
&\leq & C A(r)^{\frac{N}{N-2}}\int_{r}^{R}s^{N-3}ds \leq C (R^{N-2}-r^{N-2})A(r)^{\frac{N}{N-2}}.
\end{eqnarray*}
Thus we obtain
\begin{eqnarray}\label{5.1.9}
a(R)-(\frac{r}{R})^{N-2}a(r)&\leq& C \Big(1- (\frac{r}{R})^{N-2}\Big)A(r)^{\frac{N}{N-2}}\nonumber\\
&\leq& C A(r)^{\frac{N}{N-2}}.
\end{eqnarray}
From \eqref{5.1.9} we have
\begin{eqnarray*}
a(R') &=& (\frac{r}{R'})^{N-2}a(r) + C A(r)^{\frac{N}{N-2}}\\
&\leq & (\frac{r}{R})^{N-2}a(r) + C A(r)^{\frac{N}{N-2}} \leq \Big((\frac{r}{R})^{N-2} + C A(r)^{\frac{2}{N-2}} \Big)A(r).
\end{eqnarray*}
for any $R'\geq R$. Taking a supremum with respect to $R'$, we get
\begin{eqnarray}\label{5.1.10}
A(R) \leq \Big((\frac{r}{R})^{N-2} + C A(r)^{\frac{2}{N-2}} \Big)A(r),\ \mbox{for } R\geq r.
\end{eqnarray}
Next we fix $q\in (\frac{(N-2)^2}{N}, N-2)$ arbitrarily. We apply an iteration technique due to Giaquinta \cite[Lemma 2.1]{MG}. We take $q<\tilde{q}<N-2$ and put $\theta= 2^{\frac{1}{N-2-\tilde{q}}} >1$. Since $A(r)\to 0$ as $r\to \infty$, there exists $r_0>0$ such that
\begin{eqnarray}\label{5.1.11}
C A(r_0)^{\frac{2}{N-2}}\leq \Big(\frac{1}{\theta}\Big)^{N-2}.
\end{eqnarray}
Applying \eqref{5.1.10} with $R=\theta r_0$ and $r=r_0$, we have from \eqref{5.1.11} that
\begin{eqnarray}\label{5.1.12}
A(\theta r_0)&\leq& \Big(\Big(\frac{1}{\theta}\Big)^{N-2} + C A(r_0)^{\frac{2}{N-2}}\Big)A(r_0) \nonumber\\
&\leq& 2 \Big(\frac{1}{\theta}\Big)^{N-2}A(r_0) = \Big(\frac{1}{\theta}\Big)^{\tilde{q}}A(r_0).
\end{eqnarray}
Using \eqref{5.1.12} iteratively, we obtain
$$A(\theta^k r_0)\leq \Big(\frac{1}{\theta}\Big)^{k\tilde{q}}A(r_0)\ \mbox{ for } k\in \N.$$

Now we choose large $k\in \N$ so that $q< \frac{k}{k+1}\tilde{q}$ holds. Then for $r_0< \theta^k r_0<r \leq \theta^{k+1}r_0$, we have
$$A(\theta^k r_0)\leq \Big(\frac{r_0}{r} \Big)^{\frac{k \tilde{q}}{k+1}}A(r_0) < \Big(\frac{r_0}{r} \Big)^{q}A(r_0).$$
Since $A(r)$ is non-increasing, it follows that
$$A(r)\leq \Big(\frac{r_0}{r}\Big)^q A(r_0)\ \mbox{ for } r\geq \theta^k r_0.$$
Moreover since $A(r)$ is bounded on $[0, \theta^k r_0]$, there exists $C>0$ such that
$$A(r)\leq C r^{-q}\ \mbox{ for } r\geq 0.$$
Then from \eqref{5.1.8}, we obtain
$$b(r)\leq C r^{-\frac{qN}{N-2}}.$$
Since $N-2< \frac{qN}{N-2}$, it follows that $r^{N-3}b(r)\in L^1(0,\infty)$. Thus from \eqref{5.1.3}, we have
$$a(r)\leq Cr^{-(N-2)},\ A(r)\leq C r^{-(N-2)}.$$
Using \eqref{5.1.8} again , we obtain
\begin{eqnarray}\label{5.1.13}
b(r) \leq Cr^{-N}.
\end{eqnarray}
By the definition of $b(r)$, we also have
\begin{eqnarray}\label{5.1.14}
f(v_0(r))^{p+1}\leq C r^{-2N}.
\end{eqnarray}

Now we integrate \eqref{5.1.3} over $(0,r)$. Then from $a(0)=0$, we have
\begin{eqnarray*}
r^{N-2}a(r) &=& \int_0^r\Big(s^{N-2}a(s) \Big)'ds\\
&=& 2(N-1)\int_0^r s^{N-3}b(s)ds \leq 2(N-1)\int_0^{\infty} s^{N-3}b(s)ds \\
&=& \frac{4(N-1)}{p+1}\int_0^{\infty} s^{2N-3}f(v_0(s))^{p+1}ds\\
&=& \frac{4(N-1)}{(p+1)|S^{N-1}|}\int_{\R^N}|x|^{N-2}f(v_0(x))^{p+1}dx=:(C^{\ast})^2.
\end{eqnarray*}
By the definition of $a(r)$, we obtain
$$r^{2(N-1)}\Big(|v_0'(r)|^2 +\frac{2}{p+1}f(v_0(r))^{p+1} \Big)\leq (C^{\ast})^2.$$
Thus it follows that
$$|v_0'(r)|\leq C^{\ast}r^{-(N-1)},$$
$$v_0(r)\leq \int_0^{\infty}|v_0'(s)|ds\leq \frac{C^{\ast}}{N-2}r^{-(N-2)}.$$
We recall that $K(r)=\frac{1}{(N-2)|S^{N-1}|}r^{-(N-2)}$. Thus putting
$$(C_{\infty})^2:= \frac{4(N-1)}{p+1}\int_{\R^N}|x|^{N-2}f(v_0)^{p+1}dx,$$
we obtain
$$v_0(r)\leq C_{\infty}K(r),\ C_{\infty}K'(r)\leq v_0'(r).$$
Finally from \eqref{5.1.3} and \eqref{5.1.13} we have
\begin{eqnarray*}
r^{2(N-1)}\Big(|v_0'(r)|^2 +\frac{2}{p+1}f(v_0(r))^{p+1} \Big) &=& (C^{\ast})^2 - 2(N-1)\int_r^{\infty} s^{N-3}b(s)ds\\
&=& (C^{\ast})^2 - O(r^{-2}).
\end{eqnarray*}
Thus from \eqref{5.1.14} we obtain
$$v_0'(r)\leq C_{\infty}K'(r)(1- O(r^{-2})).$$
By integration, we also get
$$v_0(r)\geq C_{\infty}K(r)(1- O(r^{-2})).$$
\end{proof}

\end{document}